\documentclass[11pt]{amsart}
\usepackage[hmargin=3cm,vmargin=3cm]{geometry}
\usepackage{combelow} 
\usepackage[latin1,utf8]{inputenc} 
\usepackage[english]{babel}
\usepackage[backend=bibtex,doi=false,url=false,isbn=false,style=alphabetic,maxnames=6]{biblatex}
\usepackage{etex}
\AtBeginBibliography{\small}

\usepackage{amssymb,amsfonts,amsthm,amsmath,amscd,latexsym}
\usepackage[nohug]{diagrams}
\usepackage{palatino, euscript}
\usepackage{graphicx}
\usepackage[all]{xy}
\SelectTips{cm}{}

\usepackage{tikz, tikz-cd}
\usetikzlibrary{matrix, arrows, calc}
\tikzset{frontline/.style={preaction={draw=white,-,line width=6pt}},}  

\usepackage[dvips]{epsfig}
\usepackage{pinlabel}
\usepackage{psfrag}

\newcommand{\ig}[2]{\vcenter{\xy (0,0)*{\includegraphics[scale=#1]{fig/#2}} \endxy}}

\usepackage{enumitem}

\RequirePackage{color}
\definecolor{references}{rgb}{0,0,1}

\RequirePackage[pdftex,
colorlinks = true,
urlcolor = references, 
citecolor = references, 
linkcolor = references, 
]
{hyperref}

\newtheorem{thm}{Theorem}[section]
\newtheorem{lemma}[thm]{Lemma}
\newtheorem{theorem}[thm]{Theorem}
\newtheorem{observation}[thm]{Observation}

\newtheorem{proposition}[thm]{Proposition}

\newtheorem{corollary}[thm]{Corollary}

\newtheorem*{prop*}{Proposition}
\newtheorem*{lemma*}{Lemma}

\theoremstyle{definition}

\newtheorem{definition}[thm]{Definition}

\newtheorem{notation}[thm]{Notation}

\newtheorem{example}[thm]{Example}

\newtheorem{problem}[thm]{Problem}

\theoremstyle{remark}
\newtheorem{remark}[thm]{Remark}

\numberwithin{equation}{section}

\def\CB{{\mathbf C}}  

\def\EB{{\mathbf E}}

\def\KB{{\mathbf K}}    \def\KC{{\mathcal{K}}}

\def\PB{{\mathbf P}}

\def\AS{{\EuScript A}}
\def\BS{{\EuScript B}}
\def\CS{{\EuScript C}}

\def\NS{{\EuScript N}}

\def\a{\alpha}
\def\b{\beta}

\def\d{\delta}

\let\phi=\varphi

\usepackage{bbm}
\def\C{{\mathbbm C}}
\def\N{{\mathbbm N}}
\def\R{{\mathbbm R}}
\def\Z{{\mathbbm Z}}

\def\1{\mathbbm{1}}
\newcommand{\one}{\1}

\newcommand{\HH}{\operatorname{HH}}
\newcommand{\smMatrix}[1]{\left[\begin{smallmatrix}#1\end{smallmatrix}\right]}

\def\gbimod{{\text{-gbimod}}}

\newcommand{\Hom}{\operatorname{Hom}}
\newcommand{\End}{\operatorname{End}}
\newcommand{\Ext}{\operatorname{Ext}}

\newcommand{\inv}{^{-1}}

\newcommand{\Ch}{\operatorname{Ch}}

\newcommand{\Cone}{\operatorname{Cone}}
\newcommand{\SBim}{\mathbb{S}\operatorname{Bim}}

\newcommand{\FT}{\operatorname{FT}}
\newcommand{\Br}{\operatorname{Br}}

\newcommand{\Sym}{\operatorname{Sym}}
\newcommand{\poly}{\mathbf{p}}
\newcommand{\knotpoly}{\mathbf{g}}
\newcommand{\linkpoly}{\mathbf{f}}

\newcommand{\uHom}{\underline{\Hom}}
\newcommand{\uEnd}{\underline{\End}}
\newcommand{\tw}{\operatorname{tw}}
\newcommand{\KR}{\operatorname{KR}}
\newcommand{\HY}{\operatorname{HY}}
\newcommand{\unk}{?}
\newcommand{\ring}{\Z}
\newcommand{\invs}{\operatorname{inv}}
\renewcommand{\BS}{\mathbb{BS}}
\author{Matthew Hogancamp}

\address{Northeastern University}
\email{m.hogancamp@northeastern.edu}

\author{Anton Mellit}\address{University of Vienna}\email{anton.mellit@univie.ac.at}

\title{Torus link homology}

\begin{document}

\begin{abstract}
We compute the triply graded Khovanov-Rozansky homology of a family of links, including positive torus links and $\Sym^l$-colored torus knots.
\end{abstract}

\maketitle

\setcounter{tocdepth}{1}
\tableofcontents

\section{Introduction}
\label{s:intro}


In this paper we compute the triply graded Khovanov-Rozansky homology of the torus link $T(m,n)$ with $m,n\in \Z_{\geq 0}$.  These homologies have been the subject of numerous conjectures over the past decade \cite{ORS12,GorskyCatalan,GORS12,GorNeg15}.

Recent years have seen a rapid development of technology which has proven very useful in the study of KR homology, particularly KR homology of torus links.  First, in \cite{HogSym-GT}, the first named author constructed a complex of Soergel bimodules $\KB_n$ (a categorical analogue of a renormalized Young symmetrizer) which facilitates the computation of the ``stable limit'' of KR homologies of $T(m,n)$ as $m\to\infty$ (there is a second stable limit, studied in \cite{AbHog17} using different techniques).  In \cite{ElHog16a} the first author and Ben Elias showed how the same complexes $\KB_n$ can be used to compute KR homologies of many links, with the flagship example being $T(n,n)$ for arbitrary $n\geq 0$. In \cite{Hog17b} the first author applied the same technique to compute KR homologies of $T(n,nk)$ and $T(n,nk\pm 1)$ for $n,k\geq 0$.  Finally in \cite{MellitTorus-pp}, the second named author computed $T(m,n)$ for $m,n\geq 0$ coprime, again using the technique from \cite{ElHog16a}.  The ensuing recursions exactly parallel the recursions appearing in the earlier work of the second author and Erik Carlsson on the Shuffle Theorem \cite{CarMel-jams,MellitRational-pp}.



In this paper we reinterpret and generalize the main idea in \cite{MellitRational-pp} to compute the homology of $T(m,n)$ without the restriction that $m,n$ be coprime (but retaining the restriction that $m,n$ be positive), generalizing both \cite{Hog17b} and \cite{MellitRational-pp}.

We also compute the homology of $T(m,n)$ in which one of the link components is $\Sym^l$-colored (with all other components carrying the standard color). In particular, we obtain the $\Sym^l$-colored triply graded homology of torus knots.

\subsection{Main results}
\label{ss:intro main results}
If $L\subset \R^3$ is an oriented link then we let $H_{\KR}(L)$ denote the triply graded Khovanov-Rozansky homology of $L$ (see \S \ref{ss:shifts} and \S \ref{sss:normalization} for conventions concerning normalization and gradings), with integer coefficients.    The main result of this paper is a recursive computation of $H_{\KR}(T(m,n))$ with $m,n\geq 0$.  The intermediate steps in the recursion are indexed by certain pairs of binary sequences $v,w$.

\begin{definition}\label{introdef:the polys}
Let $v\in \{0,1\}^{m+l}$ and $w \{0,1\}^{n+l}$ be binary sequences with $|v|=|w|=l$.  Here $|v|=v_1+\cdots+v_{m+l}$ is the number of ones.   Let $\poly(v,w)\in \N[q,t^{\pm 1},a,(1-q)\inv]$ denote the unique family of polynomials, indexed by such pairs of binary sequences, satisfying
\begin{enumerate}\setlength{\itemsep}{3pt}
\item $\poly(\emptyset,0^n) = \left(\frac{1+a}{1-q}\right)^n$ and $\poly(0^m,\emptyset) = \left(\frac{1+a}{1-q}\right)^m$.
\item $\poly(v1,w1)=(t^l+a)\poly(v,w)$, where $|v|=|w|=l$.
\item $\poly(v0,w1)=\poly(v,1w)$.
\item $\poly(v1,w0)=\poly(1v,w)$.
\item $\poly(v0,w0)=t^{-l}\poly(1v,1w)+q t^{-l}\poly(0v,0w)$, where $|v|=|w|=l$.
\end{enumerate}
\end{definition}

It is not hard to see that the $\poly(v,w)$ are well-defined (see \S \ref{sss:statement}).   

\begin{theorem}\label{thm:intromain}
If $m,n\geq 0$ the triply graded KR homology of $T(m,n)$ is free over $\ring$ of graded rank $\poly(0^m,0^n)=\frac{1}{1-q}\poly(10^{m-1},10^{n-1})$.
\end{theorem}

\begin{example}
The first example which is not computed in either of the papers \cite{ElHog16a,Hog17b,MellitTorus-pp} is $T(4,6)$.  Theorem \ref{thm:intromain} says that $H_{\KR}(T(4,6))$ is free over $\Z$, of graded rank
\begin{eqnarray*}
\operatorname{gdim}(H_{\KR}(T(4,6))) & = & \poly(0000,000000)\\
&=& \frac{t^{-8}(1+a)}{(1-q)^2}
\bigg(- q^8 t - q^7 t^2 - q^6 t^3 - q^5 t^4 - q^4 t^5 - q^3 t^6 - q^2 t^7 - q t^8\\
&&+ q^8 + q^7 t + q t^7 + t^8\\
& &  + q^6 t - q^4 t^3 - q^3 t^4 + q t^6\\
&& + q^5 t + 2 q^4 t^2 + 2 q^3 t^3 + 2 q^2 t^4 + q t^5) \\
&+& 
a(-q^7 t - q^6 t^2 - q^5 t^3 - q^4 t^4 - q^3 t^5 - q^2 t^6 - q t^7\\
&& + q^7 - q^5 t^2 - q^4 t^3 - q^3 t^4 - q^2 t^5 + t^7\\
&& + q^6 + q^5 t - q^4 t^2 - 2 q^3 t^3 - q^2 t^4 + q t^5 + t^6\\
&& + q^5 +3 q^4 t + 3 q^3 t^2 + 3 q^2 t^3 + 3 q t^4 + t^5\\
&&  + q^3 t + q^2 t^2 + q t^3) \\
&+&
a^2(-q^5 t - q^4 t^2 - q^3 t^3 - q^2 t^4 - q t^5\\
&&+q^5 - q^3 t^2 - q^2 t^3 + t^5\\
&&+ q^4 + q^3 t + q t^3 + t^4\\
&& + q^3 + 2 q^2 t + 2 q t^2 + t^3)\\
&+&
a^3(-q^2t - qt^2 + q^2 + qt + t^2)\bigg)
\end{eqnarray*}
\end{example}

We use formal variables $Q,A,T$ to represent the three gradings on Hochschild cohomology of complexes of Soergel bimodules, and set $q:=Q^2$, $a:=AQ^{-2}$, $t:=T^2Q^{-2}$ (see \S \ref{ss:shifts}).  Then since $t$ involves even powers of $T$ (which represents (co)homological degree) we have the following as a corollary.

\begin{corollary}\label{cor:parity}
The triply graded KR homology of $T(m,n)$ is supported in even homological degrees when $m,n\geq 0$
\end{corollary}
\begin{remark}
This statement is false for negative torus links, and fails already for $T(2,-4)$.  See \S \ref{sss:negative} for a brief discussion of negative torus links.
\end{remark}

The polynomials $\poly(v,w)$ are in general not the graded dimensions Khovanov-Rozansky homologies of any links (unless $|v|\leq 1$).  Rather, they appear as graded dimensions of some special complexes of Soergel bimodules, which we explain next.

For each integer $n\geq 1$ we let $\BS_n$ denote the category of Bott-Samelson bimodules over $\ring$ (see \S \ref{ss:SBim and Rouquier}).  For each braid $\b\in \Br_n$ we have the \emph{Rouquier complex} $F(\b)$, which is a complex in $\BS_n$, well-defined up to homotopy equivalence.  Hochschild cohomology $\HH$ gives a functor from $\BS_n$ to the category of bigraded $\ring$-modules, so $\HH(F(\b))$ is a complex of bigraded $\ring$-modules (overall such a gadget has three gradings).  The triply graded Khovanov-Rozansky homology of an oriented link $L\subset \R^3$ is isomorphic (up to a shift in the trigrading) to the homology $H(\HH(F(\b))$ where $\b$ is a braid representative of $L$.

If $m,n\in \Z$ are positive, then the torus link $T(m,n)=T(-m,-n)\subset \R^3$ can be described as the closure of the braid depicted below:
\begin{equation}\label{eq:Xmn}
X_{m,n} \ \ := \ \ \begin{minipage}{1.2in}
\labellist
\small
\pinlabel $\underbrace{ \ \ \ \ \ \ \ \  }_m$ at 9 -7
\pinlabel $\underbrace{ \ \ \ \ \ \ \ \ \ \ \   }_n$ at 55 -7
\endlabellist
\begin{center}\includegraphics[scale=.85]{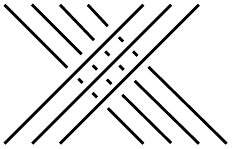}\end{center} 
\end{minipage}
\end{equation}\vskip9pt
\noindent (see \S \ref{sss:Tmn diagram}). The negative torus links $T(m,-n)=T(-m,n)$ can be described similarly, by taking the mirror image (or inverse) of the braid above.

We consider the following special family of complexes of Soergel bimodules, indexed by pair of binary sequences.  Let $v\in \{0,1\}^{m+l}$ and $w\in \{0,1\}^{n+l}$ satisfy $|w|=|v|=l$.  Let $\a_v$ and $\b_w$ be the \emph{shuffle braids} associated to $v,w$ (\S \ref{sss:shuffles}), and consider the complex $\CB(v,w)\in \KC^b(\BS_{n+m+l})$ depicted graphically by
\[
\CB(v,w) \ \ := \ \ \begin{minipage}{.9in}
\labellist
\pinlabel $m$ at 0 -5
\pinlabel $n$ at 30 -5
\pinlabel $l$ at 54 -5
\pinlabel $\a_v$ at  45  76
\pinlabel $\KB_l$ at  57  47
\pinlabel $\b_w$ at 45 23
\endlabellist
\begin{center}\includegraphics[scale=.8]{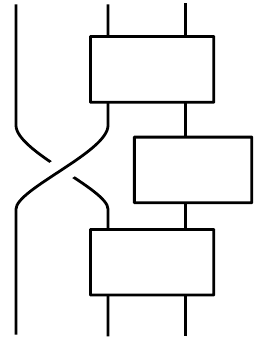}\end{center} 
\end{minipage},
\]
\vskip8pt\noindent
where $\KB_l\in \KC^b(\BS_l)$ is the categorified normalized Young symmetrizer from \cite{HogSym-GT}.  See \S \ref{ss:projectors} for recollections concerning $\KB_l$, and \S \ref{ss:diagrams} for an explanation of the diagrammatic notation.

By construction, $\CB(0^m,0^n)$ is the Rouquier complex associated to the braid $X_{m,n}$, so the homology of $\HH(\CB(0^m,0^n))$ is isomorphic to $H_{\KR}(T(m,n))$ up to a shift.  Theorem \ref{thm:intromain} is a special case of the following.

\begin{theorem}\label{thm:intromain 2}
The complex $\HH(\CB(v,w))$ is homotopy equivalent to the free triply graded $\ring$-module of graded dimension $\poly(v,w)$, with zero differential.
\end{theorem}

As a byproduct of our computation, we also obtain the Khovanov-Rozansky homology of a family of non-torus-links.  Precisely, if $v\in\{0,1\}^{m+1}$ and $w\in \{0,1\}^{n+1}$ satisfy $|v|=|w|=1$, then $\CB(v,w)$ is the Rouquier complex of the braid $(\one_n\sqcup \a_v)X_{m,n}(\one_m\sqcup \b_w)$, up to a tensor factor of the form $\one_{n+m}\sqcup \KB_1$. Proposition 4.12 of $\HH(\CB(v,w))$ tells us that
\[
\HH((\one_n\sqcup \a_v)X_{m,n}(\one_m\sqcup \b_w)) \ \simeq \ \HH(\CB(v,w))\otimes_{\ring} \ring[x]
\]
where $x$ is a formal variable of degree $(2,0,0)$ (written multiplicatively $\deg(x)=Q^2=q$).

\begin{corollary}
If $|w|=|v|=1$, the Khovanov-Rozansky homology of the link represented by $(\one_n\sqcup \a_v)X_{m,n}(\one_m\sqcup \b_w)$ is a free triply graded $\ring$-module of graded dimension $\frac{1}{1-q}\poly(v,w)$.\qed
\end{corollary}

We also obtain a result on the colored homology of torus links.  The details of colored homology are technical, so we omit them from this introduction.

\begin{theorem}\label{thm:intromain 3}
Consider the torus link $T(m,n)$ in which one component is labeled with the representation $\Sym^l(V)$, with the remaining components labeled with the standard representation $V$.  Then the triply graded homology of the resulting colored link $(T(m,n),\sigma)$ is a free triply graded $\ring$-module of dimension
\[
\prod_{i=1}^l\frac{1}{1-q t^{1-i}}\poly(1^l 0^{ml -l}, 1^l 0^{nl-l})
\]
\end{theorem}


\subsection{Open problems}
\label{ss:open probs}
It remains to compare our results with conjectures \cite{ORS12}.  Note that Theorem 5 in \cite{ORS12} calculates the cohomologies of Hilbert schemes relevant to  $T(m,n)$ only when $m,n$ are assumed coprime.

\begin{problem}
Extend the computations in \cite{ORS12} to the non coprime case, and compare with our computation of $H_{\KR}(T(m,n))$.
\end{problem}

In fact the original conjectures of \cite{ORS12} involve not just positive torus links, but arbitrary algebraic links (it is known that these are all iterated cables of torus links).  

\begin{problem}
Compute the triply graded Khovanov-Rozansky homology of algebraic links and compare with conjectures in \cite{ORS12}.
\end{problem}


\begin{remark}
All of the conjectures in \cite{GORS12} are stated with the assumption that $m,n$ are coprime. It would be interesting to generalize these to the link case (or, optimistically, to the case of arbitrary algebraic links) and compare with known computations of $H_{\KR}$.
\end{remark}

Let $L=L_1\cup\cdots\cup L_r$ be an $r$ component link.  Let $x_i,\theta_i$  ($i=1,\ldots,r)$ be formal variables of tridegree $\deg(x_i) = (2,0,0)$ and $\deg(\theta_i)=(2,-1,0)$ (written multiplicatively as $\deg(x_i)=q$, $\deg(\theta_i)=a$; see \S \ref{sss:triply graded cxs}).  We regard $\theta_i$ as odd variables, so notation such as $\ring[\mathbf{x},\boldsymbol{\theta}]$ denotes ``super-polynomial'' ring which is polynomial in the $x_i$ and exterior in the $\theta_i$.  The homology $H_{\KR}(L)$ is a well-defined isomorphism class of triply graded module over $\ring[\mathbf{x},\boldsymbol{\theta}]$.

\begin{problem}\label{prob:HKR as module}
Compute $H_{\KR}(T(m,n))$ as a triply-graded module over $\ring[\mathbf{x},\boldsymbol{\theta}]$.
\end{problem}

\begin{remark}
If $L=L_1=:K$ is a knot then
\[
H_{\KR}(K) \cong H_{\KR}^{\text{red}}(K)\otimes\ring[x_1,\theta_1],
\]
as a triply graded $\ring[x_1,\theta_1]$-module, where $H_{\KR}^{\text{red}}(K)$ denotes the \emph{reduced} homology of $K$.
\end{remark}

There is a more structured link invariant (deformed, or ``$y$-ified'', Khovanov-Rozansky homology) denoted $\HY(L)$ \cite{GorHog17}.  This deformed homology is a module over $\ring[\mathbf{x},\mathbf{y},\boldsymbol{\theta}]$, where $y_1,\ldots,y_r$ are even variables of degree $\deg(y_i)=(-2,0,2)$ (written multiplicatively as $\deg(y_i)=t$).

\begin{remark}
If $L=K$ is a knot, then
\[
\HY(K)\cong H_{\KR}^{\text{red}}(K)\otimes_{\ring} \ring[x_1,y_1,\theta_1]
\]
as a triply graded module over $\ring[x_1,y_1,\theta_1]$.
\end{remark}

Now, let $m,n,r$ be non-negative integers with $m,n$ coprime.  Then $T(m,n)$ is a knot and $T(mr,nr)$ is an $r$-component link where each component is a copy of $T(m,n)$.   There is a natural link-splitting map (see Corollary 4.14 in \cite{GorHog17})
\begin{equation}\label{eq:Tmn splitting map}
\HY(T(mr,nr))\rightarrow \HY(T(m,n))^{\otimes r} = H_{\KR}^{\text{red}}(T(m,n))\otimes \ring[\mathbf{x},\mathbf{y},\boldsymbol{\theta}].
\end{equation}
The splitting map here has degree zero because the components $T(mr,nr)$ can be unlinked by a sequence of positive-to-negative crossing changes.  Results in this paper show that $H_{\KR}(T(m,n))$ is supported in even cohomological degrees, so Theorem 4.21 in \cite{GorHog17} says that the map in \eqref{eq:Tmn splitting map} is injective.

\begin{problem}\label{prob:image of splitting}
Compute the image of $\HY(T(mr,nr))$ inside $\HY(T(m,n))^{\otimes r}$.
\end{problem}
\begin{remark}
One of the major results of \cite{GorHog17} solves Problem \ref{prob:image of splitting} in the special case $m=1$, with coefficients in $\C$.  In this case $T(1,n)=U$ is the unknot, with $H_{\KR}^{\text{red}}(U;\C)= \C$, and
\[
\HY(T(r,nr);\C)\subset \C[x_1,\ldots,x_r,y_1,\ldots,y_r]\otimes\Lambda[\theta_1,\ldots,\theta_r]
\]
is the ideal generated by the sign component with respect to the $S_r$-action which simultaneously permutes the three sets of variables.
\end{remark}
\begin{remark}
The solution of Problem \ref{prob:image of splitting} would compute $\HY(T(mr,nr))$ as a module over $\ring[\mathbf{x},\mathbf{y},\boldsymbol{\theta}]$, and it would also compute the undeformed homology $H_{\KR}(T(mr,nr))$ as a module over $\ring[\mathbf{x},\boldsymbol{\theta}]$, via
\[
H_{\KR}(T(mr,nr)) \cong \HY(T(mr,nr)) \Big/ (y_1,\ldots,y_r)\HY(T(mr,nr)),
\]
thereby also solving Problem \ref{prob:HKR as module}.
\end{remark}

\subsubsection{Negative torus links}
\label{sss:negative}
If $T(m,n)$ is the closure of the braid $X_{m,n}$ from \eqref{eq:Xmn}, then the negative torus link $T(m,-n)=T(-m,n)$ is the closure of $X_{m,n}\inv$.  The complexes which compute $H_{\KR}$ are dual to one another, as complexes of $R$-modules:
\[
\HH(X_{m,n}\inv) \cong \Hom_R(\HH(X_{m,n},R))
\]
up to a regrading (see Corollary 1.12 in \cite{GHMN-pp}.  Since we compute $\HH(X_{m,n})$ only as a complex of $\ring$-modules, we are unable to make computations for negative torus links.

Note also that in degree zero, Hochschild cohomology is just $\HH^0:=\Hom_{R\otimes R}(R,-)$.   The torus link $T(m,n)$ is also the closure of the braid $\b_{m,n}:=(\sigma_1\cdots\sigma_{m-1})^n$, and the full twist braid is $\FT_m:=\b_{m,m}$ acts as a sort of Serre functor (Theorem 1.1 in \cite{GHMN-pp}), from which it follows that
\[
\HH^0(\b_{m,-n})\simeq \uHom(\HH^0(\b_{m,n-m}), R_m). 
\]
In other words, certain questions for negative torus links can be translated into questions for positive torus links.  This is one very compelling reason why one might be interested in the structure of $\HH(\b)$ as a complex of $R$-modules.

\subsection{Organization}

In \S \ref{s:prelims} we set up notation and recall some essential background.  \S \ref{ss:complexes} concerns basics of complexes. Particularly important is the notion of a one-sided twisted complex and Lemma \ref{lemma:simplifications}, which allows us to simplify one-sided twisted complexes up to homotopy. In \S \ref{ss:SBim and Rouquier} we recall Soergel bimodules and Rouquier complexes.  This includes Hochschild cohomology (\S \ref{sss:HH}) and the Markov moves (\S \ref{sss:markov}).  Finally \S \ref{ss:projectors} briefly recalls the essential properties of the complexes $\KB_n$, first constructed in \cite{HogSym-GT}.

Section \S \ref{s:torus links} is the heart of the paper.  In \S \ref{ss:diagrams} we set up diagrammatic notation which will be heavily used in our main constructions and computations.  We also discuss diagrams for torus links (\S \ref{sss:Tmn diagram}).  In \S \ref{ss:the complexes} we introduce the complexes $\CB(v,w)$ and state the main theorem concerning $\HH(\CB(v,w))$ (Theorem \ref{thm:main thm}).  In \S \ref{ss:computations} we prove Theorem \ref{thm:main thm}.

The short \S \ref{s:colored} sketches the definition of $\Sym^l$-colored triply graded Khovanov-Rozansky homology and explains how $\Sym^l$-colored homology of torus knots arises as a special case of $\HH(\CB(v,w))$ (see Theorem \ref{thm:colored}).

Finally, in \S \ref{s:comparison} we compare the computations in this paper with those in \cite{Hog17b} (see Theorem \ref{thm:comparison}).

\subsection*{Acknowledgements}
The first author was supported by NSF grant DMS 1702274.  The authors would also like to thank Eugene Gorsky, Mikhail Mazin, and Monica Vazirani for their interest and comments on an earlier draft.

\section{Preliminaries}
\label{s:prelims}
\subsection{Complexes}
\label{ss:complexes}
Let $\AS$ an additive category.  We let $\Ch(\AS)$ be the category of (co)chain complexes
\[
\cdots \rightarrow X^k\rightarrow X^{k+1}\rightarrow\cdots
\]
and degree zero chain maps.  We always adopt the cohomological conventions for gradings of complexes,  and henceforth we will omit the prefix ``co-''.  We let $\KC(\AS)$ denote the homotopy category of $\Ch(\AS)$, with the same objects, but morphisms regarded up to chain homotopy.  Superscripts $+,-,b$ will denote full subcategories of complexes $X$ with $X^k=0$ for $k\ll 0$, $k\gg 0$, and $k$ outside a finite set, respectively.

For $X$ a complex and $k\in \Z$ let $X[k]$ denote the complex with $X[k]^l=X^{k+l}$ and $d_{X[k]}=(-1)^k d_X$.  In particular $[1]$ shifts $X$ to the left and negates the differential.

Associated to two complexes $X,Y\in \Ch(\AS)$ we have the hom complex $\uHom^{\Z}_{\Ch(\AS)}(X,Y)$, which in degree $k$ is
\[
\uHom^{k}_{\Ch(\AS)}(X,Y):=\prod_{i\in \Z}\Hom_{\AS}(X^i,Y^{i+k}),
\]
with differential given by the super-commutator
\[
f\mapsto [d,f]:=d_{Y}\circ f - (-1)^{|f|}f\circ d_{X}.
\]

Suppose $X=(X,d_X)$ is a chain complex with differential $d_X$, and let $\a\in \uEnd^\Z(X,X)$ be a degree $1$ element satisfying the \emph{Maurer-Cartan} equation
\[
[d,\a]+\a\circ \a = 0
\]
Then $(d_X+\a)^2=0$, and we can consider $X$ with the ``twisted differential'' $d_X+\a$, denoted
\[
\tw_\a(X):=(X,d_X+\a).
\]

\begin{remark}
Any chain complex $X=(X,d)$ can be written as $X=\tw_{d}(\bigoplus_{k\in \Z} X^k[-k])$.
\end{remark}

\begin{definition}\label{def:one sided} A \emph{one-sided twisted complex} is a complex of the form $\tw_\a(\bigoplus_{i\in S} X_i)$ where:
\begin{enumerate}
\item $X_i\in \Ch(\AS)$ are complexes indexed by a finite poset $S$.
\item the component $\a_{ij}\in \uHom^1(X_j,X_i)$ is zero unless $j<i$.
\end{enumerate}
\end{definition}

The following is the main technical tool for simplifying the complexes appearing in this paper.
\begin{lemma}\label{lemma:simplifications}
Suppose $S$ is a finite poset, and let $X_i,Y_i\in \Ch(\AS)$ be complexes indexed by $i\in S$.   Then any family of homotopy equivalences $X_i\simeq Y_i$ ($i\in S$) induces a (one-sided) twist $\b$ acting on $\bigoplus_i Y_i$ and a homotopy equivalence $\tw_\b(\bigoplus_i Y_i)\simeq \tw_\a(\bigoplus_i X_i)$.
\end{lemma}
In other words, the Maurer-Cartan element $\a$ can be transferred from $\bigoplus_i X_i$ to $\bigoplus_iY_i$, so that the resulting twisted complexes are homotopy equivalent.

\begin{remark}
We can also allow infinite posets in the statement of Lemma \ref{lemma:simplifications}. There are actually two kinds of infinite one-sided twisted complexes, those of the form $\tw_\a(\bigoplus_{i\in S} X_i)$, and those of the form $\tw_\a(\prod_{i\in S}X_i)$.  In the direct sum (respectively direct product) case, the statement of Lemma \ref{lemma:simplifications} requires that for each element $i\in S$ there are only finitely many $j\in S$ with $j>i$ (respectively $j<i$).
\end{remark}





\begin{notation}\label{notn:cone}
Given complexes $X_0,X_1\in \Ch(\AS)$ and a degree 1 chain map $f\in \uHom^1(X_0,X_1)$ we have a twisted complex of the form $\tw_\a(X_0\oplus X_1)$ where $\a=\smMatrix{0&0\\f &0}$.  Such twisted complexes will be denoted by
\[
\left(X_0\buildrel f\over \rightarrow X_1\right)
\]
If $g:Y_0\rightarrow Y_1$ is a degree zero chain map, then the one-sided twisted complex $(Y_0[1]\buildrel g\over\longrightarrow Y_1)$ is the usual \emph{mapping cone} of $g$.
\end{notation}

\subsection{Gradings and shifts}
\label{ss:shifts}

Let $\CS^{\Z\times \Z}(\ring)$ denote the category of $\Z\times \Z$-graded complexes of $\ring$-modules.  An object of this category is a pair $(X,d)$ where $X=\bigoplus_{i,j}X^{i,j}$ is a $\Z\times \Z$-graded $\ring$-module and $d$ is a degree $(0,1)$ endomorphism of $X$ satisfying $d^2=0$.    Morphisms in $\CS^{\Z\times \Z}(\ring)$ are by definition degree zero $\ring$-linear maps which commute with the differentials.



\newcommand{\rk}{\operatorname{rk}}

If $X$ is a bigraded $\ring$-module then we write its \emph{Poincar\'e series} or \emph{graded rank} as
\[
\operatorname{grk}(X) := \sum_{i,j\in \Z} Q^iT^j \rk(X^{i,j}).
\]

The formal variables $Q$ and $T$ (and monomials therein) will also be regarded as the grading shift functors,
\[
Q(X)^{i,j} := X^{i-1,j},\qquad\qquad T(X)^{i,j} := X^{i,j-1}.
\]
  If $X$ is equipped with a differential $d_X$ then $Q^i T^j(X)$ is a complex with differential $(-1)^j d_X$ (the sign is conventional).

\begin{remark}
The complex $Q^iT^j(X)$ would traditionally be written as $X(-i)[-j]$.
\end{remark}

\begin{notation}\label{notn:sums of shifts}
We extend the notation $Q^iT^j X $ to allow non-negative integral linear combinations of monomials.   Given $f(Q,T)=\sum_{i,j}r_{ij} Q^iT^j$ with $r_{ij}\in \Z$ and $X\in \CS^{\Z\times \Z}(\ring)$, then we let $f(Q,T) X $ denote the complex
\[
f(Q,T) X := \bigoplus_{i,j} Q^iT^j X^{\oplus r_{ij}}.
\]
\end{notation}

\subsubsection{Triply graded complexes}
\label{sss:triply graded cxs}
Let $\CS^{\Z\times \Z\times \Z}(\ring)$ denote the category of $\Z\times \Z\times \Z$-graded complexes of $\ring$-modules.   An object of this category is a pair $(X,d)$ where $X$ is a $\Z\times \Z\times \Z$-graded $\ring$-module and $d$ is a degree $(0,0,1)$ differential.  The morphisms in $\CS^{\Z\times \Z\times \Z}(\ring)$ are degree $(0,0,0)$ $\ring$-linear maps which commute with the differentials.

An ungraded $\ring$-module can be regarded as a triply graded $\ring$-module supported in degree $(0,0,0)$.  We let $Q,A,T$ denote the shifts in tridegree, so that the triply graded module $X$ can be written $X = \bigoplus_{i,j,k}Q^iA^jT^k (X^{i,j,k})$.

As the notation suggests, in this paper, bigraded $\ring$-modules can be regarded as trigraded $\ring$-modules, supported in degrees $\Z\times\{0\}\times \Z$.  As in the bigraded setting, the shift $T$ introduces a sign in all differentials: if $X$ is equipped with a degree $(0,0,1)$ differential $d_X$ then $Q^iA^jT^k(X)$ is a complex with differential appearing with the conventional sign $(-1)^k d_X$.

The \emph{Poincar\'e series} or \emph{(tri)graded rank} of a trigraded $\ring$-module $X$ is
\[
\operatorname{grk}(X) = \sum_{i,j,k\in \Z} Q^i A^j T^k \operatorname{rk}(X^{i,j,k}).
\]

If $f(Q,A,T)$ is a Laurent polynomial with non-negative integer coefficients and $X\in \CS^{\Z\times \Z\times \Z}(\ring)$, then $f(Q,A,T)(X)$ is defined in a manner analogous to Notation \ref{notn:sums of shifts}.

\begin{notation}\label{notn:degrees}
If $X$ is a trigraded $\ring$-module and $x\in X$ is trihomogeneous of tridegree $\deg(x)=(i,j,k)$ then we also write $\deg(x) = Q^iA^j T^k$. 
\end{notation}

It is often convenient to work with the formal variables $q,t,a$ defined below:
\begin{equation}\label{eq:qat}
q:=Q^2,\qquad\qquad t:=T^2Q^{-2},\qquad \qquad a:=AQ^{-2}.
\end{equation}
Thus, a monomial $q^i a^j t^k$ can refer to a grading shift functor acting on $\CS^{\Z\times\Z\times \Z}(\ring)$, or the degree of a trihomogeneous element in a  trigraded $\ring$-module.

\subsection{Soergel bimodules and Rouquier complexes}
\label{ss:SBim and Rouquier}
We very briefly recall some background concerning Soergel bimodules, omitting many details, mostly for the purposes for setting up notation.

Our results on Khovanov-Rozansky homology hold over the integers (and over any ring of coefficients by extension of scalars).  For this reason we do not really discuss Soergel bimodules as is usually meant, but rather Bott-Samelson bimodules.  When the ring of coefficients is sufficiently nice (e.g.~an infinite field of characteristic $\neq 2$) the category of Soergel bimodules is the idempotent completion of the category of Bott-Samelson bimodules, by definition.


For $n\in \Z_{\geq 1}$  we let $\BS_n$ denote the monoidal category of Bott-Samelson bimodules associated to $S_n$ with its $n$-dimensional realization $\ring^{\oplus n}$, and we let $\Ch(\BS_n)$ denote the category of complexes over $S_n$ with morphisms degree zero chain maps.

More precisely, let $R_n:=\ring[x_1,\ldots,x_n]$, thought of as a graded ring via $\deg(x_i)=2$.   Let $R_n\gbimod$ denote the category of graded $R_n,R_n$-bimodules, with degree zero $R_n$-bilinear maps as morphisms.   When the index $n$ is understood we will simply write $R=R_n$.

For $i=1,\ldots,n-1$ there is a distinguished bimodule $B_i := R\otimes_{R^{s_i}} R(1)$ where $s_i=(i,i+1)$ denotes the simple transposition in $S_n$ and $R^{s_i}\subset R$ is the subalgebra of $s_i$-invariant polynomials.  Also $(1) = Q\inv$ is the grading shift which places $1\otimes 1$ in degree $-1$.

A \emph{Bott-Samelson bimodule} is any bimodule isomorphic to a direct sum of shifts of bimodules of the form $B_{i_1}\otimes_R\cdots\otimes_R B_{i_r}$; these form a full subcategory of $R_\gbimod$, denoted $\BS_n$.  By convention, the trivial bimodule $R$ is a Bott-Samelson bimodule (corresponding to the empty tensor product).

\begin{remark}
Most of the subtleties in Soergel bimodules arise when discussing direct summands of Bott-Samelson bimodules (for instance calculating the Grothendieck group $K_0$ of this category is quite subtle in general).  In this paper such subtleties never arise, because all of the relevant constructions (for instance Rouquier complexes and the subsequent Markov moves, defined below) take place within the homotopy category of complexes of Bott-Samelson bimodules.
\end{remark}

\begin{remark}
All of the constructions and results below are valid with any ring of coefficients, since homotopy equivalences of complexes remain homotopy equivalences after extension of scalars.  When the ring of coefficients is sufficiently nice, then results of Soergel's apply and the inclusion
\[
\KC^b(\BS_n)\hookrightarrow \KC^b(\SBim_n)
\]
is an equivalence of categories.  Thus, there is no essential loss in restricting to Bott-Samelson bimodules.
\end{remark}

\subsubsection{Rouquier complexes}
\label{sss:rouquier}
Let $\Br_n$ be the braid group on $n$ strands.  For each $\b\in \Br_n$ we have the \emph{Rouquier complex} $F(\b)\in \KC^b(\BS_n)$, well-defined up to homotopy equivalence, defined as follows.  If $\sigma_i\in \Br_n$ denotes the elementary braid generator (a positive crossing relating strands $i$ and $i+1$) then we define
\[
F(\sigma_i) := B_i(-1)\rightarrow \underline{R},\qquad \qquad  F(\sigma_i\inv) = \underline{R}\rightarrow B(1).
\]

\begin{remark}
In the literature it is common to work with a different normalization, related to ours by $Q^{-e}T^eF(\b)$ where $e$ is the signed number of crossings in $\b$ (number of positive crossings $\sigma_i$ minus number of negative crossings $\sigma_i$).

Our chosen normalization will help make computations later in the paper cleaner.  For instance the positive Markov II move is satisfied with no additional shift \eqref{eq:Markov II}, and the complexes $\KB_l$ absorb braids with no additional shifts \eqref{eq:crossingAbsorption}.
\end{remark}

\subsubsection{Hochschild cohomology}
\label{sss:HH}
The Hochschild cohomology of a graded $R,R$-bimodule $B$ is a bigraded $\ring$-module $\HH^{\cdot,\cdot}(B)$ satisfying
\[
\bigoplus_{i\in \Z}\HH^{i,j}(B) = \Ext^j_{R\otimes R}(R, B).
\]

\begin{remark}
Actually the Hochschild cohomology of a graded $R,R$-bimodule is a bigraded $R$-module, but in this paper we ignore the $R$-module structure, and view $\HH(X)$ as a bigraded $\ring$-module.   
\end{remark}

\begin{notation}
In this paper we consider Hochschild cohomology exlusively, and never consider Hochschild homology, so $\HH(B)$ will always mean Hochschild cohomology of a bimodule.
\end{notation}

Since $\HH$ is an additive functor it can be extended to a functor on the level of complexes $\Ch(R\gbimod)\rightarrow \CS^{\Z\times \Z\times\Z}(\ring)$.   In other words, $\HH$ of a complex $X\in \Ch(R\gbimod)$ is obtained by applying $\HH$ term-wise
\[
\HH(X) = \begin{diagram}
\cdots  &\rTo^{\HH(d)} & \HH(X^k)  &\rTo^{\HH(d)} & \HH(X^{k+1})  &\rTo^{\HH(d)} &\cdots.
\end{diagram}
\]

Alternatively, if $X = (X,d_X)$ is a complex of $R,R$-bimodules then we first ignore the differential and regard $X$ as the direct sum of its chain objects $X=\bigoplus_k T^k(X^k)$.  Each $X^k$ is a graded $R_n,R_n$-bimodule, hence $X$ can be regarded as a bigraded $R_n,R_n$-bimodule.  The Hochschild cohomology $\HH(X)=\bigoplus_{i,j,k}Q^i A^j T^k(\HH^{i,j}(X^k))$ is then a triply graded $\ring$-module.  Because Hochschild cohomology is functorial, $\HH(X)$ is equipped with a differential $\HH(d_X)$ of degree $(0,0,1)$.




\begin{notation}
If $X\in \Ch(\BS_n)$ and $Y\in \Ch(\BS_m)$, we will write $X\sim Y$ if $\HH(X)\simeq \HH(Y)$.  Similarly, given $f(Q,A,T)\in \N[Q^{\pm},A^{\pm},T^{\pm}]$ we say $X\sim f(Q,A,T)Y$ if $\HH(X)\simeq f(Q,A,T)\HH(Y)$.
\end{notation}

\subsubsection{Markov moves}
\label{sss:markov}
The identities in this section are well-known; see \cite{Kras10} for Markov moves over $\Z$ (alternate proofs can be found in \cite{HogSym-GT}, \S 3.3).

If $X\in \KC(\BS_n)$ then we have the Markov I move:
\begin{equation}
\HH(F(\b)\otimes X \otimes F(\b\inv))\simeq \HH(X)\qquad \text{for all }\b\in \Br_n,
\end{equation}
and the Markov II move:
\begin{equation}\label{eq:Markov II}
\HH((X\sqcup \one_1) \otimes F(\sigma_n))\simeq \HH(X),  \qquad \HH((X\sqcup \one_1) \otimes F(\sigma_n\inv))  \simeq Q^{-4}AT \HH(X).
\end{equation}

We will also need the following
\begin{equation}\label{eq:tr of id}
\HH(X\sqcup \one_1) \cong \HH(X)\otimes_{\ring} \ring[x,\theta]  \cong \frac{1+a}{1-q} \HH(X)
\end{equation}
Here, $\theta$ is a formal odd variable of degree $a$ and $x$ is a formal even variable of degree $q$, so $\ring[x,\theta]$ is polynomial in $x$ and exterior in $\theta$.

\subsubsection{Normalization}
\label{sss:normalization}
There is a group homomorphism $\Br_n\rightarrow \Z$ sending $\sigma_i^{\pm}\mapsto \pm 1$.  The image of $\b$ will be denoted $e(\b)$.  The number $n$ is called the \emph{braid width} or \emph{braid index}, and $e(\b)$ is the \emph{writhe} or \emph{exponent sum}.

To obtain an honest link invariant, we normalize $\HH(F(\b))$ by applying a shift $\Sigma$ which depends on the braid width $n$, the writhe $e$, and the number $c$ of components of $L=\hat{\b}$.
 
One normalization which works well is
\[
H_{\KR}^{\operatorname{norm}}(L) = (Q^{-4}AT)^{(e+c-n)/2} H(\HH(F(\b))).
\]
Note that $e+c-n$ is always even (exercise), so that the above shift makes sense.

\subsection{The complexes $\KB_n$}
\label{ss:projectors}
In \cite{HogSym-GT} the first named author constructed complexes $\KB_n\simeq \KC^b(\BS_n)$ satisfying the following properties:
\begin{subequations}
\begin{equation}\label{eq:crossingAbsorption}
\KB_n \otimes F(\b)\simeq \KB_n \simeq F(\b)\otimes \KB_n)
\end{equation}
\begin{equation}\label{eq:Kmarkov}
\HH((X\sqcup \one_1)\otimes \KB_{n+1}) \simeq (t^{n}+a)\HH(X).
\end{equation}
\begin{equation}\label{eq:Krecursion}
(\one_1\sqcup \KB_n)\otimes L_{n+1} \simeq t^{-n}\left(\KB_{n+1} \rightarrow q(\one_1\sqcup \KB_n)\right),
\end{equation}
\end{subequations}
This is true for all complexes $X\in \BS_n$ and all braids $\b\in \Br_n$.  In the last line we introduced the braid $L_{n+1}:=F(\sigma_1\cdots\sigma_{n-1}\sigma_n^2\sigma_{n-1}\cdots\sigma_1)$.

In \cite{ElHog16a} it was shown how these relations yield a calculus for computing  $\HH(X)$ (at least partially) when $X$ is a Rouquier complex tensored with some $\one_a\sqcup \KB_b\sqcup \one_c$.

\section{Torus link homology}
\label{s:torus links}

In this section we introduce some useful diagrammatic shorthand.  We then define a special family of complexes of Soergel bimodules and compute their $\HH$ recursively using this diagrammatic shorthand.  As a special case we obtain $H_{\KR}$ of positive torus links, generalizing \cite{Hog17b,MellitTorus-pp}.


\subsection{Diagrams for braids, links, and complexes}
\label{ss:diagrams}
%
A strand with the label $n$ will denote $n$ parallel copies of that strand (drawn in the plane of the page):
\begin{equation}\label{eq:basiccables}
\begin{minipage}{.4in}
\labellist
\small
\pinlabel $n$ at -8 15
\endlabellist
\begin{center}\includegraphics[scale=.85]{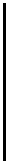}\end{center} 
\end{minipage}
\ \ = \ \
\begin{minipage}{.5in}
\labellist
\small
\pinlabel $\underbrace{\ \ \ \ \ \ \ }_n$ at 10 -6
\endlabellist
\begin{center}\includegraphics[scale=.85]{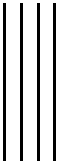}\end{center} 
\end{minipage},
\qquad \qquad
\begin{minipage}{.8in}
\labellist
\small
\pinlabel $n$ at 0 10
\pinlabel $m$ at 50 10
\endlabellist
\begin{center}\includegraphics[scale=.85]{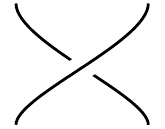}\end{center} 
\end{minipage}
\ \ = \ \ 
\begin{minipage}{.8in}
\labellist
\small
\pinlabel $\underbrace{ \ \ \ \ \ \ \ \  }_n$ at 9 -7
\pinlabel $\underbrace{ \ \ \ \ \ \ \ \ \ \ \   }_m$ at 55 -7
\endlabellist
\begin{center}\includegraphics[scale=.85]{fig/cabledCrossing}\end{center} 
\end{minipage} 
\end{equation}
\vskip7pt
We will also introduce diagrams which represent the identity braid on $n+m$ strands (regrouped):
\[
\begin{minipage}{.8in}
\labellist
\small
\pinlabel $m$ at -7 30
\pinlabel $n$ at 44 30
\pinlabel $m+n$ at 0 5
\endlabellist
\begin{center}\includegraphics[scale=.85]{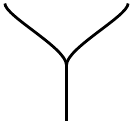}\end{center} 
\end{minipage} 
\ \ :=\ \ 
\begin{minipage}{.8in}
\labellist
\small
\pinlabel $m$ at 0 15
\pinlabel $n$ at 30 15
\endlabellist
\begin{center}\includegraphics[scale=.85]{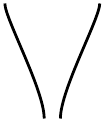}\end{center} 
\end{minipage},
\qquad\qquad
\begin{minipage}{.8in}
\labellist
\small
\pinlabel $m$ at -7 5
\pinlabel $n$ at 44 5
\pinlabel $m+n$ at 0 30
\endlabellist
\begin{center}\includegraphics[scale=.85]{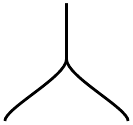}\end{center} 
\end{minipage} 
\ \ :=\ \ 
\begin{minipage}{.8in}
\labellist
\small
\pinlabel $m$ at 0 18
\pinlabel $n$ at 30 18
\endlabellist
\begin{center}\includegraphics[scale=.85]{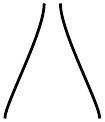}\end{center} 
\end{minipage} 
\]

\subsubsection{Diagrams for torus links}
\label{sss:Tmn diagram}
\begin{proposition}
If $m,n\geq 0$ then the torus link $T(m,n)$ is the closure of the braid
\[
X_{m,n} \ := \ 
\begin{minipage}{.8in}
\labellist
\small
\pinlabel $m$ at 0 10
\pinlabel $n$ at 50 10
\endlabellist
\begin{center}\includegraphics[scale=.85]{fig/crossing_2}\end{center} 
\end{minipage}.
\]
Graphically this is
\begin{equation}\label{eq:symmetric diagram}
T(m,n) \  \ \ = \ \ \ 
\begin{minipage}{1.3in}
\labellist
\small
\pinlabel $m$ at -5 40
\pinlabel $n$ at 45 40
\pinlabel $m+n$ at 3 100
\endlabellist
\begin{center}\includegraphics[scale=.85]{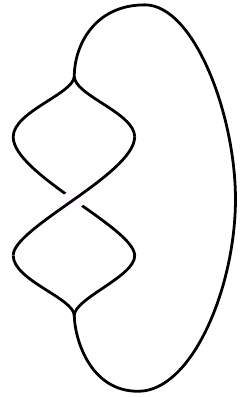}\end{center} 
\end{minipage}
\end{equation}
\end{proposition}
\begin{proof}[Sketch of proof]
The link depicted on the right-hand side of \eqref{eq:symmetric diagram} can be embedded in the surface
\[
\ig{.8}{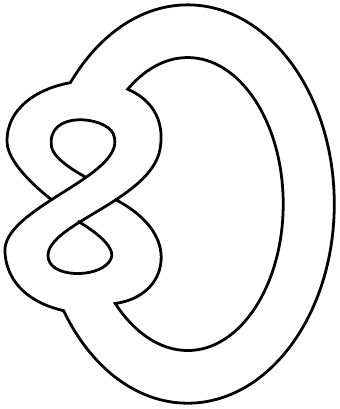} \  \ \ \ \ \ \simeq \ \ \ \ \ \ \ \ig{.8}{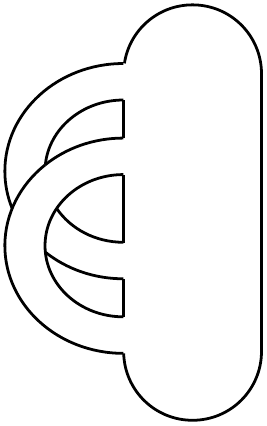}\ \ ,
\]
which is a standardly embedded 2-dimensional torus in $\R^3$, minus an open disk.  Thus, the braid closure of $X_{m,n}$ is a torus link (positive since all $X_{m,n}$ is clearly a positive braid).  The class in homology $H^1(S^1\times S^1)$ represented by the closure of $X_{m,n}$ can be calculated by counting intersections with the arcs $a$ and $b$ below (co-cores of the indicated 1-handles):
\[
\begin{minipage}{1.3in}
\labellist
\small
\pinlabel $a$ at 20 15
\pinlabel $b$ at 20 110
\endlabellist
\begin{center}\includegraphics[scale=.85]{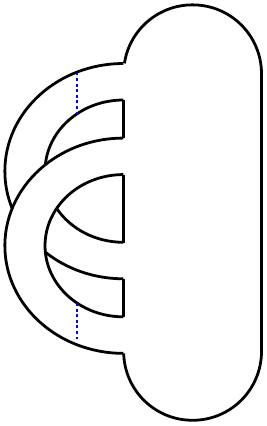}\end{center} 
\end{minipage}.
\]
The numbers of these intersections are $m$ and $n$, respectively.  
\end{proof}

\subsubsection{Representing complexes}
\label{sss:diagrams for cxs}
We will denote complexes in $\KC^b(\BS_n)$ and certain categorical operations on complexes diagrammatically.

A complex $C\in \KC^b(\BS_n)$ will indicated by a diagram
\[
C \ \ \leadsto \ \  
\begin{minipage}{.8in}
\labellist
\small
\pinlabel $n$ at 16 5
\pinlabel $C$ at 22 34
\endlabellist
\begin{center}\includegraphics[scale=.6]{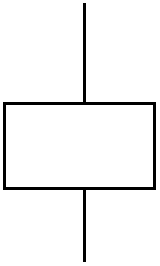}\end{center} 
\end{minipage}.
\]
We have two categorical operations on objects of $\BS_n$, or complexes.  First, we have the tensor product $X\otimes Y=X\otimes_R Y$, defined for $X,Y\in \KC^b(\BS_n)$, and we have the external tensor product $X\sqcup Y:=X\otimes_\Z Y$.  These operations are indicated diagrammatically by
\[
X\otimes Y = \begin{minipage}{.8in}
\labellist
\small
\pinlabel $n$ at 18 5
\pinlabel $X$ at 24 72
\pinlabel $Y$ at 24 31
\endlabellist
\begin{center}\includegraphics[scale=.6]{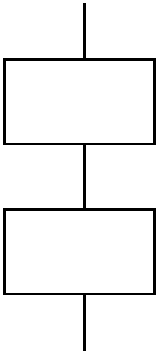}\end{center} 
\end{minipage},
\qquad\qquad 
X\sqcup Y \ \ = \ \ 
\begin{minipage}{.5in}
\labellist
\small
\pinlabel $n$ at 18 5
\pinlabel $X$ at 22 34
\endlabellist
\begin{center}\includegraphics[scale=.6]{fig/box}\end{center} 
\end{minipage}\ \ 
\begin{minipage}{.5in}
\labellist
\small
\pinlabel $m$ at 15 5
\pinlabel $Y$ at 22 34
\endlabellist
\begin{center}\includegraphics[scale=.6]{fig/box}\end{center} 
\end{minipage}.
\]

We will also allow braids as part of our diagrams. See \eqref{eq:Dvw} for example.

\subsection{A distinguished family of complexes}
\label{ss:the complexes}
In this section we introduce the complexes whose Hochschild cohomologies will be computed recursively.  These complexes will involve the projectors $\KB_l$ and also Rouquier complexes associated to so-called ``shuffle braids'', which we recall next.

\subsubsection{Shuffle braids}
\label{sss:shuffles}
Let $v\in\{0,1\}^{r}$ be a sequence.  Let $|v|=v_1+\cdots+v_r$ be the number of ones in $v$.    Let $\pi_v\in S_r$ denote the ``shuffle permutation'' which sends $\{1,\ldots,k\}$ and $\{k+1,\ldots,k+l\}$ to the set of indices $i$ for which $v_i=0$ (respectively $v_i=1$).  Here, $l=|v|$ and $k=r-l$.  Alternatively, $\pi_v$ can be defined inductively by the following rules:
\begin{itemize}
\item if $r=1$, then $\pi_v=\one_1$ is the identity.
\item $\pi_{v1}:=\pi_v\sqcup \one_1$.
\item $\pi_{v0}:=(\pi_v\sqcup\one_1) s_{r-1}\cdots s_{r-l}$, $l=|v|$.
\end{itemize}
 Below is a closed formula for $\pi_v$:
\[
\pi_v = (s_{i_1}\cdots s_1)(s_{i_2}\cdots s_2) \cdots (s_{i_k}\cdots s_{k})
\]
where $\{i_1<i_{2}<\cdots<i_k\}\subset \{1,\ldots,r\}$ are the indices for which $v_{i_j}=0$.

\begin{definition}\label{def:shuffle braids}
If $v\in \{0,1\}^r$, then we let $\a_v\in \Br_r$ denote the positive braid lift of $\pi_v$, and we let $\b_v$ denote the positive braid lift of $\pi_v\inv$.
\end{definition}

The braids $\a_v$ satisfy the following relations:
\begin{equation}\label{eq:av from right}
\a_{v1}  \ \ = \ \ 
\begin{minipage}{.6in}
\labellist
\small
\pinlabel $\a_v$ at 18 18
\pinlabel $k$ at 8 45
\pinlabel $l$ at 28 45
\pinlabel $1$ at 42 45
\endlabellist
\begin{center}\includegraphics[scale=.8]{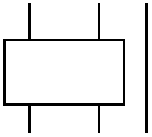}\end{center} 
\end{minipage},
\qquad
\a_{v0}  \ \ = \ \ 
\begin{minipage}{.6in}
\labellist
\small
\pinlabel $\a_v$ at 21 17
\pinlabel $k$ at 8 57
\pinlabel $1$ at 28 57
\pinlabel $l$ at 42 57
\endlabellist
\begin{center}\includegraphics[scale=.8]{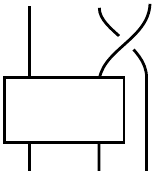}\end{center} 
\end{minipage}
\end{equation}
\begin{equation}\label{eq:av from left}
\a_{1v}  \ \ = \ \ 
\begin{minipage}{.6in}
\labellist
\small
\pinlabel $\a_v$ at 28 18
\pinlabel $k$ at 0 57
\pinlabel $1$ at 18 57
\pinlabel $l$ at 38 57
\endlabellist
\begin{center}\includegraphics[scale=.8]{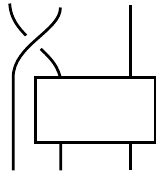}\end{center} 
\end{minipage},
\qquad
\a_{0v}  \ \ = \ \ 
\begin{minipage}{.8in}
\labellist
\small
\pinlabel $\a_v$ at 26 17
\pinlabel $1$ at 0 45
\pinlabel $k$ at 18 45
\pinlabel $l$ at 38 45
\endlabellist
\begin{center}\includegraphics[scale=.8]{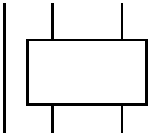}\end{center} 
\end{minipage}.
\end{equation}
where $v$ has $k$ zeroes and $l$ ones.  There are similar identities involving $\b_v$:
\begin{equation}\label{eq:bv from right}
\b_{v1}  \ \ = \ \ 
\begin{minipage}{.8in}
\labellist
\small
\pinlabel $\b_v$ at 18 18
\pinlabel $k$ at 8 -5
\pinlabel $l$ at 28 -5
\pinlabel $1$ at 42 -5
\endlabellist
\begin{center}\includegraphics[scale=.8]{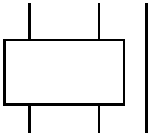}\end{center} 
\end{minipage},
\qquad
\b_{v0}  \ \ = \ \ 
\begin{minipage}{.8in}
\labellist
\small
\pinlabel $\b_v$ at 20 31
\pinlabel $k$ at 8 -5
\pinlabel $1$ at 28 -5
\pinlabel $l$ at 42 -5
\endlabellist
\begin{center}\includegraphics[scale=.8]{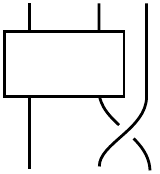}\end{center} 
\end{minipage},
\end{equation}
\begin{equation}\label{eq:bv from left}
\b_{1v}  \ \ = \ \ 
\begin{minipage}{.6in}
\labellist
\small
\pinlabel $\b_v$ at 28 31
\pinlabel $k$ at 0 -6
\pinlabel $1$ at 18 -6
\pinlabel $l$ at 38 -6
\endlabellist
\begin{center}\includegraphics[scale=.8]{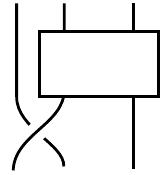}\end{center} 
\end{minipage},
\qquad
\b_{0v}  \ \ = \ \ 
\begin{minipage}{.8in}
\labellist
\small
\pinlabel $\b_v$ at 26 17
\pinlabel $1$ at 0 -6
\pinlabel $k$ at 18 -6
\pinlabel $l$ at 38 -6
\endlabellist
\begin{center}\includegraphics[scale=.8]{fig/alpha_0v}\end{center} 
\end{minipage}.
\end{equation}

\subsubsection{The complexes}
\label{sss:the complexes}

Let $v\in\{0,1\}^{m+l}$ and $w\in \{0,1\}^{n+l}$ be sequences with $|v|=l=|w|$.  Define the complexes
\[
\CB(v,w) \  \ := \ \ \left(\one_n\sqcup F(\a_v)\right)\otimes \left(F(X_{m,n}) \sqcup \KB_l \right) \otimes \left(\one_m\sqcup F(\b_w)\right),
\]
where $\a_v$ and $\b_v$ are the braids introduced in \S \ref{sss:shuffles} above.

The complex $\CB(v,w)\in \KC^b(\BS_{m+n+l})$ will be depicted diagrammatically by
\begin{equation}\label{eq:Dvw}
\CB(v,w) \ := \ 
\begin{minipage}{.8in}
\labellist
\pinlabel $m$ at 0 -5
\pinlabel $n$ at 30 -5
\pinlabel $l$ at 54 -5
\pinlabel $\a_v$ at  45  76
\pinlabel $\KB_l$ at  57  47
\pinlabel $\b_w$ at 45 23
\endlabellist
\begin{center}\includegraphics[scale=.8]{fig/Dvw}\end{center} 
\end{minipage}
\end{equation}

\subsubsection{Statement of the main theorem}
\label{sss:statement}

\begin{definition}\label{def:the polys}
Let $v\in \{0,1\}^{m+l}$ and $w\in  \{0,1\}^{n+l}$ be binary sequences with $|v|=|w|=l$.  Here $|v|=v_1+\cdots+v_{m+l}$ is the number of ones.   Let $\poly(v,w)\in \N[q,t^{\pm 1},a,(1-q)\inv]$ denote the unique family of polynomials, indexed by such pairs of binary sequences, satisfying
\begin{enumerate}\setlength{\itemsep}{3pt}
\item $\poly(\emptyset,0^n) = \left(\frac{1+a}{1-q}\right)^n$ and $\poly(0^m,\emptyset) = \left(\frac{1+a}{1-q}\right)^m$.
\item $\poly(v1,w1)=(t^l+a)\poly(v,w)$, where $|v|=|w|=l$.
\item $\poly(v0,w1)=\poly(v,1w)$.
\item $\poly(v1,w0)=\poly(1v,w)$.
\item $\poly(v0,w0)=t^{-l}\poly(1v,1w)+q t^{-l}\poly(0v,0w)$, where $|v|=|w|=l$.
\end{enumerate}
\end{definition}

\begin{lemma}\label{lemma:polys are well-defd}
The polynomials $\poly(v,w)$ are well-defined.
\end{lemma}
\begin{proof}
We prove uniqueness first, assuming existence.  
Note that relation (5) forces
\begin{equation}\label{eq:allzeroes}
\poly(0^m,0^n) = \frac{1}{1-q} \poly(10^{m-1},10^{n-1})
\end{equation}

Define a transitive, reflexive relation on binary sequences by declaring $v\leq v'$ if
\begin{enumerate}
\item $\ell(v)< \ell(v')$, or
\item $\ell(v)= \ell(v')$ and $|v|>|v'|$, or
\item $\ell(v)= \ell(v')$,  $|v|=|v'|$, and $\invs(v)\leq \invs(v')$.
\end{enumerate}
Here $\ell(v)$ denotes the length of $v$, so $v\in \{0,1\}^{\ell(v)}$, and $\invs(v)$ denotes the number of inversions of $v$, i.e.~the number of pairs of indices $i<j$ with $v_i=1$, $v_j=0$.  Then $\emptyset$ is the unique minimum sequence, and for every $v'$ there are only finitely many $v$ with $\emptyset\leq v\leq v'$.

Write $(v,w)\leq (v',w')$ if $v\leq v'$ and $w\leq w'$.    If $v$ and $w$ are not identically zero (this case is taken care of with \eqref{eq:allzeroes}) then each of the relations (1)-(5) writes $\poly(v',w')$ in terms of $\poly(v,w)$ with $(v,w)<(v',w')$.  This proves uniquess. 

Now, note that for a given pair $(v,w)$, exactly one of the rules (1)-(5) applies, hence these rules define $\poly(v,w)$ recursively (this recursion terminates because each $(v,w)$ has only finitely many predecessors).
%
\end{proof}

Recall that $\CS^{\Z\times \Z\times \Z}(\ring)$ denotes the category of triply graded complexes of $\ring$-modules (\S \ref{sss:triply graded cxs}), and we regard $\HH(X)$ as an object of $\CS^{\Z\times \Z\times \Z}(\ring)$ for any $X\in \Ch(\KC(\BS))$.
\begin{theorem}\label{thm:main thm}
The Hochschild cohomology $\HH(\CB(v,w))$ is homotopy equivalent to the free $\Z^3$-graded $\ring$-module $\poly(v,w) \ring$ with zero differential.   In particular the Poincar\'e series of $H_{\text{KR}}(T(m,n))$ equals $\poly(0^m,0^n)$.
\end{theorem}

\subsection{The computations}
\label{ss:computations}
We prove Theorem \ref{thm:main thm} by showing that $\HH(\CB(v,w))$ satisfies categorical analogues of the recursion which defines $\poly(v,w)$.

\begin{lemma}\label{lemma:Dv1w1}
Let $v\in\{0,1\}^{m+l}$ and $w\in \{0,1\}^{n+l}$ be sequences with $|v|=|w|=l$.  Then 
\[
\HH(\CB(v1,w1))\simeq (t^l+a) \HH(\CB(v,w))
\]
\end{lemma}
\begin{proof}
We have
\[
\CB(v1,w1) \ := \ 
\begin{minipage}{1.2in}
\labellist
\pinlabel $m$ at 0 -5
\pinlabel $n$ at 30 -5
\pinlabel $l$ at 52 -5
\pinlabel $1$ at 72 -5
\pinlabel $\a_v$ at  45  76
\pinlabel $\KB_{l+1}$ at  60  48
\pinlabel $\b_w$ at 45 22
\endlabellist
\begin{center}\includegraphics[scale=.8]{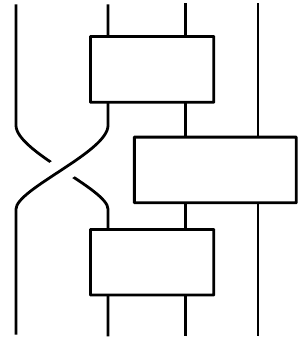}\end{center} 
\end{minipage},
\]\vskip7pt
\noindent and an application of \eqref{eq:Kmarkov} proves the Lemma.
\end{proof}

\begin{lemma}\label{lemma:Dv0w1}
Let $v\in\{0,1\}^{m+l}$ and $w\in \{0,1\}^{n+l}$ be sequences with $|v|=|w|-1=l$.  Then 
\[
\HH(\CB(v0,w1))\simeq  \HH(\CB(v,1w))
\]
\end{lemma}
\begin{proof}
We have
\[
\CB(v0,w1) \ := \ 
\begin{minipage}{1.4in}
\labellist
\pinlabel $n$ at 0 -5
\pinlabel $1$ at 16 -5
\pinlabel $m$ at 41 -5
\pinlabel $l$ at 70 -5
\pinlabel $1$ at 86 -5
\pinlabel $m$ at 0 124
\pinlabel $n$ at 40 124
\pinlabel $l+1$ at 68 124
\pinlabel $1$ at 90 124
\pinlabel $\a_v$ at  59  98
\pinlabel $\KB_{l+1}$ at  71  52
\pinlabel $\b_w$ at 59 25
\endlabellist
\begin{center}\includegraphics[scale=.8]{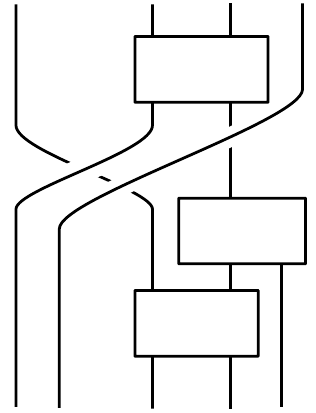}\end{center} 
\end{minipage}
\ \ \simeq \ \ 
\begin{minipage}{1.4in}
\labellist
\pinlabel $m$ at 0 -5
\pinlabel $1$ at 16 -5
\pinlabel $n$ at 51 -5
\pinlabel $l$ at 75 -5
\pinlabel $1$ at 103 -5
\pinlabel $m$ at 0 124
\pinlabel $n$ at 45 124
\pinlabel $l+1$ at 70 124
\pinlabel $1$ at 103 124
\pinlabel $\a_v$ at  63  98
\pinlabel $\KB_{l+1}$ at  76  76
\pinlabel $\b_w$ at 63 25
\endlabellist
\begin{center}\includegraphics[scale=.8]{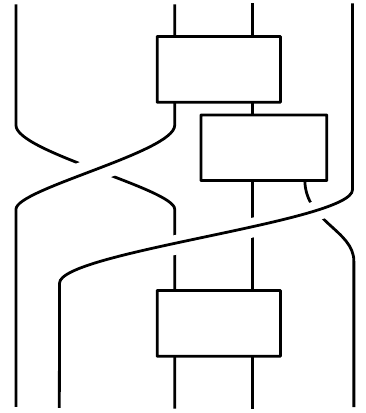}\end{center} 
\end{minipage}.
\]\vskip7pt
\noindent After a Markov move and the absorption of $l$ positive crossings (using \eqref{eq:crossingAbsorption}), this becomes\vskip7pt
\[
\begin{minipage}{1.4in}
\labellist
\pinlabel $n$ at 0 -5
\pinlabel $1$ at 16 -5
\pinlabel $m$ at 51 -5
\pinlabel $l$ at 72 -5
\pinlabel $m$ at 0 123
\pinlabel $n$ at 50 123
\pinlabel $l+1$ at 80 123
\pinlabel $\a_v$ at  60  98
\pinlabel $\KB_{l+1}$ at  76  76
\pinlabel $\b_w$ at 60 25
\endlabellist
\begin{center}\includegraphics[scale=.8]{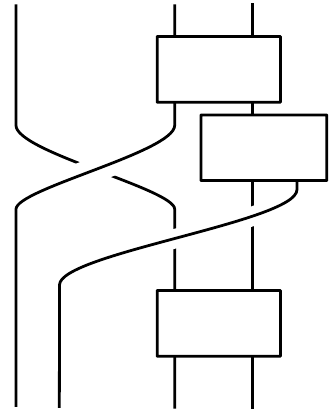}\end{center} 
\end{minipage}
\ \ \simeq \ \ 
\begin{minipage}{1.4in}
\labellist
\pinlabel $m$ at 0 -5
\pinlabel $1$ at 16 -5
\pinlabel $n$ at 55 -5
\pinlabel $l$ at 80 -5
\pinlabel $n$ at 0 124
\pinlabel $n$ at 55 124
\pinlabel $l+1$ at 80 124
\pinlabel $\a_v$ at  65  98
\pinlabel $\KB_{l+1}$ at  80  72
\pinlabel $\b_w$ at 65 25
\endlabellist
\begin{center}\includegraphics[scale=.8]{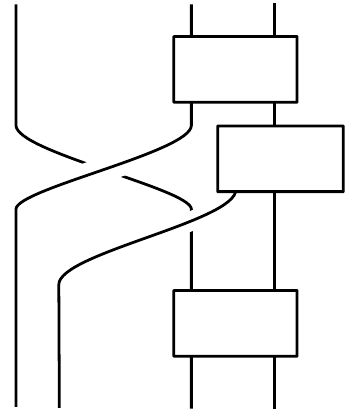}\end{center} 
\end{minipage}
 \ \ =: \ \ \CB(v,1w)
\]\vskip9pt
\noindent This proves the lemma.
\end{proof}

By symmetry we also obtain the following.

\begin{lemma}\label{lemma:Dv1w0}
Let $v\in\{0,1\}^{m+l}$ and $w\in \{0,1\}^{n+l}$ be sequences with $|v|-1=|w|=l$.  Then 
\[
\HH(\CB(v1,w0))\simeq  \HH(\CB(1v,w))
\]\qed
\end{lemma}

\begin{lemma}\label{lemma:Dv0w0}
Let $v\in\{0,1\}^{m+l}$ and $w\in \{0,1\}^{n+l}$ be sequences with $|v|=|w|=l$.  Then 
\[
\HH(\CB(v0,w0))\simeq  t^{-l}\Big(\HH(\CB(1v,1w))\rightarrow q \HH(\CB(0v,0w)) \Big)
\]
\end{lemma}
\begin{proof}
First we rewrite $\CB(v0,w0)$ by an isotopy and a Markov move:
\vskip9pt\noindent
\[
\CB(v0,w0) \ := \ 
\begin{minipage}{1.4in}
\labellist
\pinlabel $n$ at 0 145
\pinlabel $1$ at 15 145
\pinlabel $m$ at 48 145
\pinlabel $l$ at 70 145
\pinlabel $1$ at 92 145
\pinlabel $m$ at 0 -5
\pinlabel $1$ at 15 -5
\pinlabel $n$ at 48 -5
\pinlabel $l$ at 70 -5
\pinlabel $1$ at 92 -5
\pinlabel $\a_v$ at  60  113
\pinlabel $\KB_l$ at  87  70
\pinlabel $\b_w$ at 60 25
\endlabellist
\begin{center}\includegraphics[scale=.7]{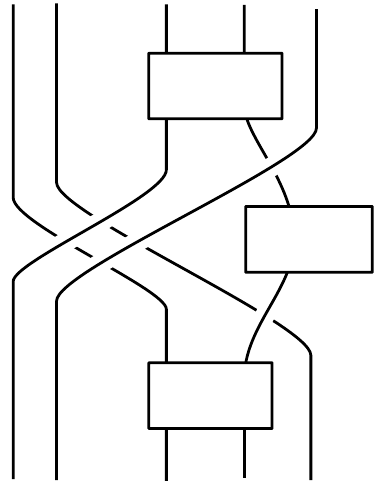}\end{center} 
\end{minipage}
\ \ \simeq \ \ 
\begin{minipage}{1.4in}
\labellist
\pinlabel $n$ at 0 145
\pinlabel $1$ at 15 145
\pinlabel $m$ at 48 145
\pinlabel $l$ at 70 145
\pinlabel $1$ at 106 145
\pinlabel $m$ at 0 -5
\pinlabel $1$ at 15 -5
\pinlabel $n$ at 48 -5
\pinlabel $l$ at 70 -5
\pinlabel $1$ at 106 -5
\pinlabel $\a_v$ at  60  113
\pinlabel $\KB_l$ at  77  77
\pinlabel $\b_w$ at 60 18
\endlabellist
\begin{center}\includegraphics[scale=.7]{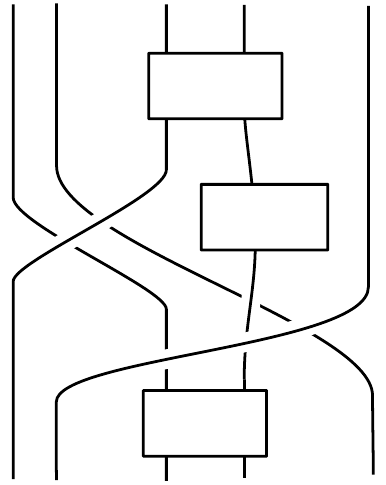}\end{center} 
\end{minipage}
\ \  \sim \ \ 
\begin{minipage}{1.2in}
\labellist
\pinlabel $n$ at 0 145
\pinlabel $1$ at 15 145
\pinlabel $m$ at 48 145
\pinlabel $l$ at 70 145
\pinlabel $m$ at 0 -5
\pinlabel $1$ at 15 -5
\pinlabel $n$ at 48 -5
\pinlabel $l$ at 70 -5
\pinlabel $\a_v$ at  60  113
\pinlabel $\KB_l$ at  77  77
\pinlabel $\b_w$ at 60 18
\endlabellist
\begin{center}\includegraphics[scale=.7]{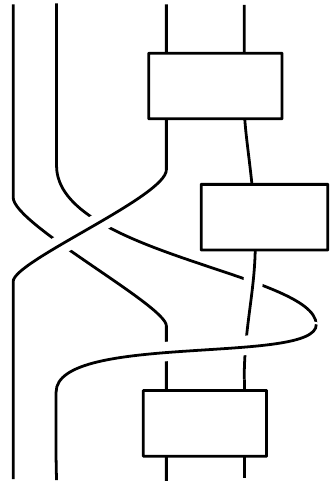}\end{center} 
\end{minipage}.
\]\vskip9pt
\noindent Next we apply \eqref{eq:Krecursion} to rewrite the tensor factor $(\one_1\sqcup \KB_l)\otimes F(\sigma_1\cdots\sigma_{l-1}\sigma_l^2\sigma_{l-1}\cdots\sigma_1)$:
\vskip9pt
\[
\begin{minipage}{1.2in}
\labellist
\pinlabel $n$ at 0 145
\pinlabel $1$ at 15 145
\pinlabel $m$ at 48 145
\pinlabel $l$ at 70 145
\pinlabel $m$ at 0 -5
\pinlabel $1$ at 15 -5
\pinlabel $n$ at 48 -5
\pinlabel $l$ at 70 -5
\pinlabel $\a_v$ at  60  113
\pinlabel $\KB_l$ at  77  77
\pinlabel $\b_w$ at 60 18
\endlabellist
\begin{center}\includegraphics[scale=.7]{fig/Dv0w0_3}\end{center} 
\end{minipage}
\ \ \simeq \ \ 
\left(
t^{-l} \begin{minipage}{1.2in}
\labellist
\pinlabel $n$ at 0 145
\pinlabel $1$ at 22 145
\pinlabel $m$ at 48 145
\pinlabel $l$ at 70 145
\pinlabel $m$ at 0 -6
\pinlabel $1$ at 22 -6
\pinlabel $n$ at 48 -6
\pinlabel $l$ at 70 -6
\pinlabel $\a_v$ at  60  113
\pinlabel $\KB_{l+1}$ at  64  65
\pinlabel $\b_w$ at 60 18
\endlabellist
\begin{center}\includegraphics[scale=.7]{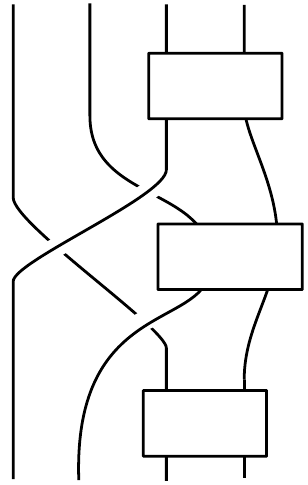}\end{center} 
\end{minipage}
\rightarrow qt^{-l} 
\begin{minipage}{1.2in}
\labellist
\pinlabel $n$ at 0 145
\pinlabel $1$ at 15 145
\pinlabel $m$ at 48 145
\pinlabel $l$ at 70 145
\pinlabel $m$ at 0 -6
\pinlabel $1$ at 15 -6
\pinlabel $n$ at 48 -6
\pinlabel $l$ at 70 -6
\pinlabel $\a_v$ at  60  113
\pinlabel $\KB_l$ at  77  77
\pinlabel $\b_w$ at 60 18
\endlabellist
\begin{center}\includegraphics[scale=.7]{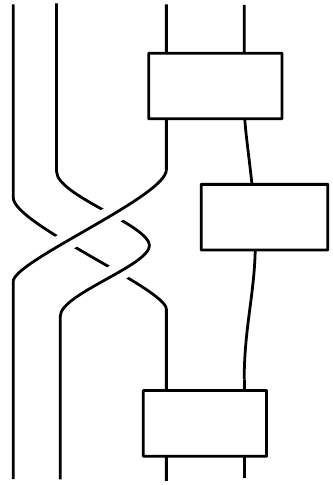}\end{center} 
\end{minipage}
\right)
\]
\vskip9pt\noindent
The first term on the right-hand side above is $t^{-l} \CB(1v,1w)$, and the second term can be manipulated by an isotopy and an inverse Markov move:\vskip7pt
\[
\begin{minipage}{1.2in}
\labellist
\pinlabel $n$ at 0 145
\pinlabel $1$ at 15 145
\pinlabel $m$ at 48 145
\pinlabel $l$ at 70 145
\pinlabel $m$ at 0 -6
\pinlabel $1$ at 15 -6
\pinlabel $n$ at 48 -6
\pinlabel $l$ at 70 -6
\pinlabel $\gamma_v$ at  60  113
\pinlabel $\KB_l$ at  77  77
\pinlabel $\b_w$ at 60 18
\endlabellist
\begin{center}\includegraphics[scale=.7]{fig/Dv0w0_5}\end{center} 
\end{minipage}
\ \ \simeq \ \ 
\begin{minipage}{1.2in}
\labellist
\pinlabel $n$ at 10 145
\pinlabel $1$ at 25 145
\pinlabel $m$ at 58 145
\pinlabel $l$ at 80 145
\pinlabel $m$ at 10 -6
\pinlabel $1$ at 25 -6
\pinlabel $n$ at 58 -6
\pinlabel $l$ at 80 -6
\pinlabel $\a_v$ at  70  113
\pinlabel $\KB_l$ at  87  70
\pinlabel $\b_w$ at 70 18
\endlabellist
\begin{center}\includegraphics[scale=.7]{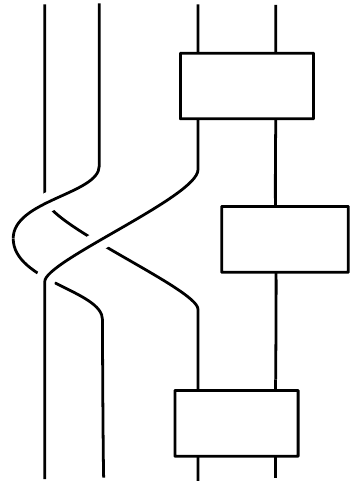}\end{center}
\end{minipage}
\ \ \sim \ \ 
\begin{minipage}{1.2in}
\labellist
\pinlabel $1$ at 0 145
\pinlabel $n$ at 23 145
\pinlabel $1$ at 49 145
\pinlabel $m$ at 75 145
\pinlabel $l$ at 101 145
\pinlabel $1$ at 0 -6
\pinlabel $m$ at 23 -6
\pinlabel $1$ at 49 -6
\pinlabel $n$ at 75 -6
\pinlabel $l$ at 101 -6
\pinlabel $\a_v$ at  89  113
\pinlabel $\KB_l$ at  100  70
\pinlabel $\b_w$ at 89 18
\endlabellist
\begin{center}\includegraphics[scale=.7]{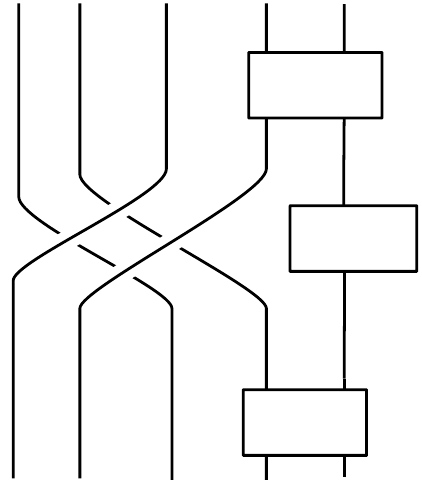}\end{center} 
\end{minipage}
\ \ = \ \ \CB(0v,0w)
\]
\vskip7pt
\noindent  This proves the lemma.
\end{proof}




\begin{proof}[Proof of Theorem \ref{thm:main thm}]
We prove the theorem by induction on $(v,w)$, using the partial order on pairs of binary sequences from the proof of Lemma \ref{lemma:polys are well-defd}.

In the base case $\CB(0^m,\emptyset) = \CB(\emptyset,0^m)=R_m$ is the identity bimodule, and $\HH(R_m)$ is calculated by repeated application of \eqref{eq:tr of id}:
\[
\HH(R_m)\cong \left(\frac{1+a}{1-q}\right)^m \ring.
\]
This proves the base case.

Suppose we wish prove the theorem for $(v',w')$ where $v'$ and $w'$ are both nonempty.  Assume by induction that the theorem holds for all $(v,w)<(v',w')$.

\textbf{Case 0. }  If $(v',w')=(0^m,0^n)$ then we have
\[
\CB(v',w') = 
\begin{minipage}{.8in}
\labellist
\small
\pinlabel $m$ at 0 10
\pinlabel $n$ at 50 10
\endlabellist
\begin{center}\includegraphics[scale=.8]{fig/crossing_2}\end{center} 
\end{minipage} 
=
\begin{minipage}{1in}
\labellist
\small
\pinlabel $1$ at 0 -6
\pinlabel $m-1$ at 24 -6
\pinlabel $1$ at 50 -6
\pinlabel $n-1$ at 73 -6
\endlabellist
\begin{center}\includegraphics[scale=.8]{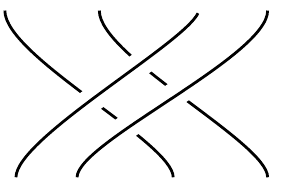}\end{center} 
\end{minipage}
\]
\vskip7pt\noindent
Now, we tensor on the right with $\one_m\sqcup \KB_1\sqcup \one_{n-1}$ and apply a Markov II move and some isotopies, obtaining:
\[
\begin{minipage}{1.1in}
\labellist
\small
\pinlabel $1$ at 0 -6
\pinlabel $m-1$ at 24 -6
\pinlabel $1$ at 50 -6
\pinlabel $n-1$ at 73 -6
\pinlabel $\KB_1$ at 50 24
\endlabellist
\begin{center}\includegraphics[scale=.8]{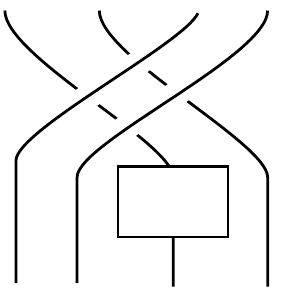}\end{center} 
\end{minipage}
\sim 
\begin{minipage}{1in}
\labellist
\small
\pinlabel $m-1$ at 0 -6
\pinlabel $1$ at 34 -6
\pinlabel $n-1$ at 60 -6
\pinlabel $\KB_1$ at 37 26
\endlabellist
\begin{center}\includegraphics[scale=.8]{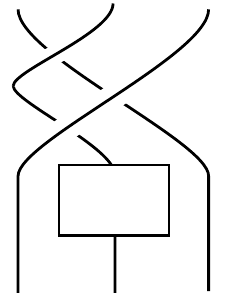}\end{center} 
\end{minipage}
\simeq
\begin{minipage}{1in}
\labellist
\small
\pinlabel $m-1$ at 24 -6
\pinlabel $1$ at 50 -6
\pinlabel $n-1$ at 73 -6
\pinlabel $\KB_1$ at 37 26
\endlabellist
\begin{center}\includegraphics[scale=.8]{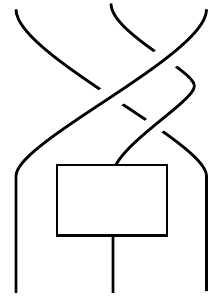}\end{center} 
\end{minipage}
\simeq
\begin{minipage}{1in}
\labellist
\small
\pinlabel $m$ at 0 -6
\pinlabel $1$ at 25 -6
\pinlabel $n$ at 50 -6
\pinlabel $\KB_1$ at 52 37
\endlabellist
\begin{center}\includegraphics[scale=.8]{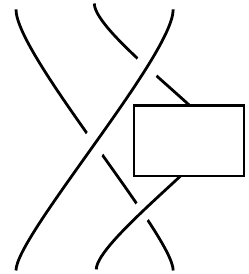}\end{center} 
\end{minipage}
=: \CB(10^{m-1},10^{n-1})
\]\vskip9pt\noindent
Thus, $\CB(0^m,0^n)$ and $\CB(10^{m-1},10^{n-1})$ are related to one another by tensoring with $\KB_1$, isotopies, and a Markov move.  The isotopies and Markov move induce homotopy equivalences after applying $\HH(\cdots)$, and so we need only consider the effect of tensoring with $\KB_1$.

Note that $(10^{m-1},10^{n-1})<(0^m,0^n)$ and so $\HH(\CB(10^{m-1},10^{n-1}))$ is supported in even homological degrees by induction.  Thus Proposition 4.12 in \cite{ElHog16a} (rather, its proof) tells us that $\HH(-)$ before and after tensoring with $\KB_1$ are related by $-\otimes_{\ring} \ring[x]$.  Precisely:
\begin{eqnarray*}
\HH(\CB(0^m,0^n)) 
&\simeq &\HH(\CB(10^{m-1},10^{n-1}))\otimes_{\ring} \ring[x] \\
&\simeq & \frac{1}{1-q}\poly(10^{m-1},10^{n-1})\ring\\
& =& \poly(0^m,0^n)\ring,
\end{eqnarray*}
where $x$ is a formal variable of degree $q$.

\textbf{Case 1. } If $(v',w')=(v1,w1)$ with $|v|=|w|=l$ then $(v,w)<(v',w')$ and
\[
\HH(\CB(v',w'))\simeq (t^l +a) \HH(\CB(v,w)) \simeq (t^l+a)\poly(v,w)\ring = \poly(v',w')\ring
\]
In the first equivalence we used Lemma \ref{lemma:Dv1w1}, in the second we used the induction hypothesis, and in the last we used the definition of $\poly(v1,w1)$.

\textbf{Case 2. } If $(v',w')=(v1,w0)$ with $|v|=|w|-1=l$ then $(1v,w)<(v',w')$ and
\[
\HH(\CB(v',w'))\simeq \HH(\CB(1v,w)) \simeq \poly(1v,w)\ring = \poly(v',w')\ring
\]
In the first equivalence we used Lemma \ref{lemma:Dv1w0}, in the second we used the induction hypothesis, and in the last we used the definition of $\poly(v1,w0)$.

\textbf{Case 3. } If $(v',w')=(v0,w1)$ with $|v|-1=|w|=l$ then the theorem holds for $(v',w')$ by symmetry (compar with Case 2). 

\textbf{Case 4. } If $(v',w')=(v0,w0)$ with $|v|=|w|=l\neq 0$ then $(1v,1w)<(v',w')$ and $(0v,0w)<(v',w')$, and
\[
\HH(\CB(v',w'))\simeq \Big(t^{-l}\HH(\CB(1v,1w))\rightarrow qt^{-l} \HH(\CB(0v,0w))\Big).
\]
by Lemma \ref{lemma:Dv0w0}.  We use the induction hypothesis to simplify each term in the right-hand side of the above, obtaining
\[
\HH(\CB(v',w'))\simeq \Big(t^{-l}\poly(1v,1w)\ring \buildrel \d\over \rightarrow qt^{-l} \poly(0v,0w)\ring\Big).
\]
The polynomials $t^{-l}\poly(1v,1w)$ and $qt^{-l} \poly(0v,0w)$ involve integer powers of $t:=T^2Q^{-2}$ and $q:=Q^2$, hence are supported in purely even homological degrees.  This forces the differential $\d$ to be zero, and we conclude
\[
\HH(\CB(v',w'))\simeq \Big(t^{-l}\poly(1v,1w)\ring \oplus qt^{-l} \poly(0v,0w)\ring\Big) = \poly(v',w')\ring.
\]
This completes the inductive step and completes the proof of Theorem \ref{thm:main thm}.
\end{proof}

\section{Colored homology of torus knots}
\label{s:colored}

\subsection{Categorified symmetrizers}
\label{ss:cat idempts}

We recall some results from \cite{HogSym-GT}.  Let $\NS\subset \BS_n$ be the full subcategory consisting of direct sums of shifts of non-trivial Bott-Samelson bimodules $B_{i_1}\otimes \cdots\otimes B_{i_m}$ with $m\geq 1$.  There is a complex $\PB_n\in \KC^-(\BS_n)$ uniquely characterized up to homotopy equivalence by:
\begin{enumerate}
\item[(P1)] $\PB_n\otimes B\simeq 0 \simeq B\otimes \PB_n$ whenever $B\in \NS$.
\item[(P2)] there is a chain map $\eta:\one_n\rightarrow \PB_n$ such that $\Cone(\eta)$ is homotopy equivalent to a complex in $\KC^-(\NS)$.
\end{enumerate}

\begin{remark}
In case $n=1$ we have $\NS=0$ (the subcategory containing only the zero bimodule) and $\eta:\one_1\rightarrow \PB_1$ is a homotopy equivalence.
\end{remark}

Let $u_k$ be formal variables of degree $(2k,2-2k)$ (written multiplicatively as $Q^{2k}T^{2-2k}$ or $q t^{1-k}$), and consider the complex $\KB_n\otimes_{\ring} \ring[u_1,\ldots,u_n]$.   In \cite{HogSym-GT} it is shown how to construct a twisted differential $d+\a$ so that the resulting the twisted complex $\tw_\a(\KB_n\otimes_{\ring} \ring[u_1,\ldots,u_n])\simeq \PB_n$.

\begin{remark}
Note that $\KB\otimes_{\ring} \ring[u_1,\ldots,u_n]$ does not live in $\KC^-(\BS_n)$, strictly speaking, since $\ring[u_1]\cong \one \oplus q\one\oplus q^2\one\oplus \cdots$ is an infinite direct sum.  Thus, in order for the constructions in this section to make sense, we close $\BS_n$ with respect to countable direct sums.
\end{remark}

\begin{lemma}
Fix integers $i,j,k\geq 0$, and let $n=i+j+k$.  Let $X\in \KC^b(\BS_n)$ be such that $\HH(X\otimes (\one_i\sqcup \KB_j\sqcup \one_k))$ is supported in even homological degrees.  Then
\begin{eqnarray*}
\HH(X\otimes (\one_i\sqcup \PB_j\sqcup \one_k))
&\simeq &\HH(X\otimes (\one_i\sqcup \KB_j\sqcup \one_k))\otimes_{\ring} \ring[u_1,\ldots,u_n]\\
&\simeq & \prod_{m=1}^n \frac{1}{1-qt^{1-m}}\HH(X\otimes (\one_i\sqcup \KB_j\sqcup \one_k)).
\end{eqnarray*}
\end{lemma}
\begin{proof}
Similar to the proof of Proposition 4.12 in \cite{ElHog16a}.
\end{proof}

\subsection{Colored homology}

One can define colored triply graded link homology using the complexes $\PB_n$, much as in \cite{CK12a}.  See also \cite{CautisColoured}.

Let $L=L_1\cup \cdots \cup L_r \subset \R^3$ be an oriented link and let $l_1,\ldots,l_r$ be non-negative integers (thought of as labeling the components of $L$).  The pair $(L,\underline{l})$ will be referred to as a \emph{colored link}.  We also say that the component $L_i$ is $l_i$-colored, or colored with the one-row Young diagram with $l_i$ boxes (or equivalently the $\operatorname{GL}_N$ representation $\Sym^{l_i}(\C^N)$ with $N\gg 0$).

\begin{remark}
Strictly speaking we should also choose a framing of $L$; this is important when discussing properly normalized link invariants, but will be ignored for now.
\end{remark}

Let $\b$ be a braid representative of $L$.  The strands of $\b$ inherit integer labels from $(L,\underline{l})$.  Now we form a triply graded complex of $\ring$-modules by the following procedure:
\begin{enumerate}
\item choose a collection of marked points on $\b$, away from the crossings, so that there is at least one marked point on each component of $L=\hat{\b}$.
\item replace an $l$-labeled strand by $l$ parallel copies of itself, and replace a marked point on such a strand by a box labeled $\PB_l$.
\item take $\HH$ of the complex represented by the diagram from (2). 
\end{enumerate}
The resulting complex will be denoted $\HH(\b,\underline{l},\Omega)$ where $\underline{l}$ represents the colors and $\Omega$ represents the chosen collection of marked points.

The following is proved using standard arguments (see \cite{CK12a}).
\begin{theorem}
The complex $\HH(\b,\underline{l},\Omega)$ depends only on the colored link represented by $(\b,\underline{l})$ up to homotopy and overall shift in the trigrading.
\end{theorem}
We let $H_{\KR}(L, \underline{l})$ denote the homology of $\HH(\b,\underline{l},\Omega)$.

\begin{theorem}\label{thm:colored}
Up to a factor of $\prod_{i=1}^l (1-qt^{1-i})\inv$ the complex $\HH(\CB(1^l0^{ml-l},1^l0^{nl-l}))$ computes the triply graded homology of the torus link $T(m,n)$ in which one component is $l$-colored and the remaining components are 1-colored.  That is to say,
\[
H_{\KR}(T(m,n), (l,\underbrace{1,\ldots,1}_{\operatorname{gcd}(m,n)-1})) \cong \prod_{i=1}^l \frac{1}{1-qt^{1-i}} H(\HH(\CB(1^l0^{ml-l},1^l0^{nl-l})))
\]
\end{theorem}
\begin{proof}

\end{proof}

\begin{example}
The $\Sym^l$-colored homology of the unknot is
\[
\prod_{i=1}^l\frac{1}{1-qt^{1-i}}\poly(1^l,1^l) = \prod_{i=1}^l \frac{t^{i-1}+a}{1-qt^{1-i}}.
\]
\end{example}
\begin{example}
The $\Sym^2$-colored homology of the trefoil has Poincar\'e series equal to
\[
\poly(0011,000011) =\frac{t^{-5} (1+a) (t + a)}{(1-q)(1-qt\inv)} (t^5 + q t^3 + q^2 t + q t^2 + a(t^3 +q t  + t^2  + q) + a^2)
\]
(up to an overall factor of the form $q^i t^j a^k$).  The reduced Poincar\'e series is obtained by dividing by the invariant of the $\Sym^2$-colored unknot: 
\[
t^5 + q t^3 + q^2 t + q t^2 + a(t^3 +q t  + t^2  + q) + a^2.
\]
This agrees with the prediction (again, up to an overall monomial) in \cite{GGSquad}, at $\mathbf{t}_r=1$, after the substitution
\[
q\mapsto \mathbf{q}^2,\qquad\quad t\mapsto \mathbf{q}^{-2}\mathbf{t}_c^{-2},\qquad\quad a\mapsto \mathbf{a}^2 \mathbf{t}_c\inv
\]
where bold letters $\mathbf{q},\mathbf{a},\mathbf{t}_c,\mathbf{t}_r$ denote the gradings in \cite{GGSquad}.

\end{example}

\section{Comparison with earlier recursions}
\label{s:comparison}

In this section we compare the recursions which define $\poly(w,v)$ with the computations in \cite{Hog17b} for $T(N,Nr)$ and $T(N,Nr+1)$.

\subsection{Admissible fillings}
The intermediate steps in the recursions in this paper are indexed by binary sequences $(v,w)$ with $|v|=|w|$, whereas those in \cite{Hog17b} are indexed by sequences $\sigma\in \{0,1,\ldots,r\}^N$ for some $r,N\geq 1$.  Our first task is to construct for each $\sigma\in\{0,1,\ldots,r\}^N$ a pair of binary sequences $(v,w)$.  The relation between $\sigma$ and $(v,w)$ is mediated by certain simple combinatorial objects, introduced below.

Fix integers $N,r\geq 1$.  Consider a filling of an $r\times N$ grid with the symbols $1,0,\ast$.  Such a filling is \emph{admissible} if it satisfies the constraints:
\begin{enumerate}
\item each row and each column may have at most one `1'.
\item each cell below a `1' is labeled with `$\ast$'.
\item all other cells are labeled '0`.
\end{enumerate}
We say that a cell is \emph{occupied} if it is labeled with a 1.  Similarly, a column is \emph{occupied} if it contains an occupied cell.

Let $T$ be such a filling.  From $T$ we construct two binary sequences $v(T)$, $w(T)$ as follows.  Label the columns of $T$ from left to right by integers $i\in \{1,\ldots,N\}$.  For each $i$ we let $v(T)_i=1$ if the $i$-th column of $T$ is occupied, and $v(T)_i=0$ otherwise.  The sequence $w(T)$ is obtained by reading the entries of $T$, from left to right, starting with the bottom row, and skipping all cells labeled `$\ast$'.

\begin{remark}
If $r=1$ then $w(\sigma)=v(\sigma)$.
\end{remark}

\begin{remark}\label{rmk:recovering filling}
Given $m$ and $r$, the filling $T$ can be recovered entirely from $w(T)$, by reading $w(T)$ from right-to-left and filling boxes from right-to-left and top-to-bottom.
\end{remark}

From $T$ we can also construct a sequence $\sigma(T)\in\{0,\ldots,r\}^N$ by letting $\sigma(T)_i$ be the number of zeroes in the $i$-th column.  This defines a bijection between admissible $r\times N$ fillings and sequences $\{0,\ldots,r\}^N$.   Given $\sigma\in \{0,1,\ldots,r\}^m$, we let $T(\sigma)$ denote the corresponding filling and $v(\sigma):=v(T(\sigma))$, $w(\sigma):=w(T(\sigma))$ be the associated binary sequences.  Note that $v(\sigma)_i=0$ if and only if $\sigma_i=r$; the sequence $w(\sigma)$ is less easily described.


\begin{example}
Let $N=4$ and $r=5$, and $\sigma:=(3,0,1,5)$.   The corresponding filling $T(\sigma)$ is 
\[
T(\sigma) \ \ = \ \ 
\begin{tabular}{|c|c|c|c|}
\hline
0&1&0&0\\
\hline
0&$\ast$ & 1 & 0 \\
\hline
0&$\ast$& $\ast$ & 0\\
\hline
1&$\ast$&$\ast$ &0\\
\hline
$\ast$&$\ast$&$\ast$ &0\\
\hline
\end{tabular}
\]
In this case $v(\sigma) = (1,1,1,0)$ and $w(\sigma)=(0,1,0,0,0,0,1,0,0,1,0,0)$.
\end{example}

\subsubsection{Rotation of fillings}
\label{sss:rotation}

Fix integers $N,r$.  We put a total order on the cells in in $r\times N$ grid so that cells $c,c'$ satisfy $c\leq c'$ if $c'$ is in a higher row than $c$ or $c,c'$ are in the same row and $c'$ is to the right of $c$.

Let $T$ be an admissible filling.  We consider the operation of rotation $T\mapsto \phi(T)$, which simply deletes the entry of $T$ in the top right and shifts all labels from their current cells to their successors.   Note that this shifts columns to the right.  The right-most column shifts up (losing its top entry) and cycles to the far left.  There are three possibilities:
\begin{enumerate}
\item the right-most column of $T$ is $1,\ast,\ldots,\ast$.   In this case $\phi(T)$ is just $T$ with the right-most column deleted.
\item the right-most column of $T$ is occupied, but its occupied cell is not in the top row. In this case $\phi(T)$ is the result of shifting the columns of $T$ cyclically to the right, and also shifting the occupied cell in the right-most column up.  
\item the right-most column of $T$ is unoccupied.  In this case rotation creates a vacancy in the bottom left cell.  This in this case we let $\phi_0(T)$ denote the result of rotating $T$ and filling the vacant bottom left cell with 0, and we let $\phi_1(T)$ denote the result of rotation and filling the vacant bottom cell with a 1.
\end{enumerate}

The effect of rotation is illustrated below:
\[
\begin{tabular}{|c|c|c|}
\hline
\unk&\unk&1\\
\hline
\unk&\unk&$\ast$\\
\hline
\unk&\unk&$\ast$\\
\hline
\unk&\unk&$\ast$\\
\hline
\end{tabular}
 \ \ \mapsto \ \
 \begin{tabular}{|c|c|c|}
\hline
\unk&\unk\\
\hline
\unk&\unk\\
\hline
\unk&\unk\\
\hline
\unk&\unk\\
\hline
\end{tabular}\:,
\]
\[
\begin{tabular}{|c|c|c|}
\hline
\unk&\unk&0\\
\hline
\unk&\unk&$0$\\
\hline
\unk&\unk&$1$\\
\hline
\unk&\unk&$\ast$\\
\hline
\end{tabular}
 \ \ \mapsto \ \
 \begin{tabular}{|c|c|c|}
\hline
0&\unk&\unk\\
\hline
1&\unk&\unk\\
\hline
$\ast$&\unk&\unk\\
\hline
$\ast$&\unk&\unk\\
\hline
\end{tabular}\:,
\]
\[
\begin{tabular}{|c|c|c|}
\hline
\unk&\unk&0\\
\hline
\unk&\unk&$0$\\
\hline
\unk&\unk&$0$\\
\hline
\unk&\unk&0\\
\hline
\end{tabular}
 \ \ \mapsto \ \
 \begin{tabular}{|c|c|c|}
\hline
0&\unk&\unk\\
\hline
0&\unk&\unk\\
\hline
0&\unk&\unk\\
\hline
0&\unk&\unk\\
\hline
\end{tabular}
\quad \text{ or } \quad
 \begin{tabular}{|c|c|c|}
\hline
0&\unk&\unk\\
\hline
0&\unk&\unk\\
\hline
0&\unk&\unk\\
\hline
1&\unk&\unk\\
\hline
\end{tabular}\:.
\]

\begin{observation}
On the level of sequences $\sigma$ rotation has the following effect:
\begin{enumerate}\label{obs:rotation of sigma}
\item $\sigma  0\mapsto \sigma$.
\item $\sigma  k \mapsto (k-1)  \sigma$ if $1\leq k\leq r-1$.
\item $\sigma  r\mapsto  r\sigma$ or $(r-1)\sigma$.
\end{enumerate}
\end{observation}

\begin{observation}\label{obs:rotation of v,w}
On the level of  binary sequences $(v,w)$ rotation has the following effect:
\begin{enumerate}
\item $(v1,w1)\mapsto (v,w)$.
\item $(v1,w0)\mapsto (1v,w)$.
\item $(v0,w0) \mapsto (0v,0w)$ or $(1v,1w)$
\end{enumerate}
\end{observation}

\subsection{Matching the recursions}

Now we are ready to match the recursions in \cite{Hog17b} with special cases of recursions appearing in this paper.
\begin{definition}
For each $\sigma\in\{0,1,\ldots,r\}^N$ define
\[
\linkpoly(\sigma):= \poly(v(\sigma),w(\sigma)),\qquad\qquad \knotpoly(\sigma) =\poly(v(\sigma),w(\sigma)0).
\]
\end{definition}

\begin{lemma}\label{lemma:comparison}
The polynomials $\linkpoly(\sigma)$ satisfy
\begin{enumerate}\setlength{\itemsep}{3pt}
\item[(L1)] $\linkpoly(\sigma 0) = (t^{l}+a)\linkpoly(\sigma)$, where $l=\#\{i\:|\: \sigma_i<r\}$.
\item[(L2)] $\linkpoly(\sigma k) =\linkpoly((k-1)\sigma)$ if $1\leq k\leq r-1$.
\item[(L3)] $\linkpoly(\sigma r) = t^{-l}\linkpoly((r-1)\sigma) + qt^{-l} \linkpoly(r\sigma)$, where $l=\#\{i\:|\: \sigma_i<r\}$.
\end{enumerate}
The polynomials $\knotpoly(\sigma)$ satisfy
\begin{enumerate}
\item[(K1a)] $\knotpoly(\sigma k 0) = (t^{l+1}+a)\knotpoly((k-1)\sigma)$, where $l=\#\{i\:|\: \sigma_i<r\}$ and $0\leq k\leq r-1$.
\item[(K1b)] $\knotpoly(\sigma r 0) = \knotpoly((r-1)\sigma)$.
\item[(K2)] $\knotpoly(\sigma k) =\knotpoly((k-1)\sigma)$ if $1\leq k\leq r-1$.
\item[(K3)] $\knotpoly(\sigma r) = t^{-l}\knotpoly((r-1)\sigma) + qt^{-l} \knotpoly(r\sigma)$, where $l=\#\{i\:|\: \sigma_i<r\}$.
\end{enumerate}
\end{lemma}
\begin{proof}
This is a trivial consequence of the definitions together with Observations \ref{obs:rotation of sigma} and \ref{obs:rotation of v,w}.  Let us illustrate this by proving (L2), (K2), and (K1b) leaving the other cases to the reader.

Let $\sigma$ be given, and let $1\leq k\leq r-1$.   Then we consider $\sigma k$ and its rotation $(k-1) \sigma$.  The associated pair of binary sequences is of the form $(v1,w0)$ and, after rotation, $(1v,w)$.  
 Then (L2) is an immediate consequence of $\poly(v1,w0) = \poly(1v,w)$.  The proof of (K2) is equally easy:
\[
\knotpoly(\sigma k) = \poly(v1,w00) = \poly(1v,w0) = \knotpoly((k-1)\sigma).
\]

The proof of (K1b) is a little more interesting.  In this case we consider $\sigma r 0$ and $(r-1)\sigma$.  The pair of binary sequences associated to $\sigma r0$ is of the form $(v01,w01)$ for some appropriate binary seqences $v,w$.  Then rotating once sends $\sigma r 0\mapsto \sigma r$ and $(v01,w01)\mapsto (v0,w0)$.  Rotating a second time (filling the vacant cell with a 1) sends $\sigma r\mapsto (r-1)\sigma$ and $(v0,w0)\mapsto (1v,1w)$.  Thus, the operation $\sigma r 0 \mapsto (r-1)\sigma$ corresponds to $(v01,w01)\mapsto (1v,1w)$.

Now, we compute:
\[
\knotpoly(\sigma r 0) = \poly(v01,w010) =  \poly(1v0,w01) = \poly(1v,1w0)  =\knotpoly((r-1)\sigma),
\]
which proves (K1b).
\end{proof}

Now recall the polynomials $f_\sigma$ and $g_\sigma$ from \cite{Hog17b}.  Let $\operatorname{rev}(\sigma)$ denote the sequence with $\operatorname{rev}(\sigma)_i = \sigma_{m+1-i}$.  Let $\invs(\sigma)$ denote the number of \emph{inversions}, i.e.~ the number of pairs of indices $i<j$ with $\sigma_i>\sigma_j$.

\begin{theorem}\label{thm:comparison}
We have $\linkpoly(\sigma) = t^{c(\sigma)} f_{\operatorname{rev}(\sigma)}$ and $\knotpoly(\sigma)=t^{c(\sigma)} g_{\operatorname{rev}(\sigma)}$ for all $\sigma\in \{0,\ldots,r\}^N$, where
\[
c(\sigma) := \invs(\sigma) + \sum_{k=1}^r \binom{\#\{i\:|\: k\leq \sigma_i\leq r\}}{2}.
\]
\end{theorem}
\begin{proof}
Lemma \ref{lemma:comparison} shows that $\linkpoly(\sigma)$ and $\knotpoly(\sigma)$ satisfy the same recursions as $f_\sigma$ and $g_\sigma$ up to powers of $t$ and reversal of $\sigma$. Keeping track of the extra powers of $t$ is tedious but straightforward.
\end{proof}

\printbibliography

\end{document}